\documentclass[11pt]{amsart}
\usepackage[a4paper, margin={1in}]{geometry}

\DeclareFontFamily{U}{matha}{\hyphenchar\font45}
\DeclareFontShape{U}{matha}{m}{n}{
      <5> <6> <7> <8> <9> <10> gen * matha
      <10.95> matha10 <12> <14.4> <17.28> <20.74> <24.88> matha12
      }{}
\DeclareSymbolFont{matha}{U}{matha}{m}{n}
\DeclareFontFamily{U}{mathx}{\hyphenchar\font45}
\DeclareFontShape{U}{mathx}{m}{n}{
      <5> <6> <7> <8> <9> <10>
      <10.95> <12> <14.4> <17.28> <20.74> <24.88>
      mathx10
      }{}
\DeclareSymbolFont{mathx}{U}{mathx}{m}{n}

\DeclareMathSymbol{\obot}         {2}{matha}{"6B}
\DeclareMathSymbol{\bigobot}       {1}{mathx}{"CB}

\usepackage{amsmath}
\usepackage{amsfonts}
\usepackage{amssymb}
\usepackage{amsthm}
\usepackage{amscd}
\usepackage{hyperref}
\usepackage{algorithmic, algorithm}
\usepackage[all]{xy}
\usepackage{stackengine}

\usepackage{color}
\usepackage[dvipsnames]{xcolor}
\usepackage{graphicx}

\usepackage{mathrsfs} 

\newcommand{\cO}{\mathcal{O}}

\DeclareMathOperator{\Q}{\mathbf{Q}}
\DeclareMathOperator{\R}{\mathbf{R}}

\DeclareMathOperator{\Z}{\mathbf{Z}}
\DeclareMathOperator{\F}{\mathbf{F}}

\DeclareMathOperator{\tr}{tr}

\DeclareMathOperator{\rank}{rank}

\DeclareMathOperator{\Vol}{Vol}

\DeclareMathOperator{\disc}{disc}

\DeclareMathOperator{\Gal}{Gal}
\DeclareMathOperator{\SL}{SL}

\DeclareMathOperator{\diag}{diag}
\DeclareMathOperator{\Hom}{Hom}

\DeclareMathOperator{\Cl}{Cl}

\theoremstyle{plain} 
\newtheorem{theorem}{Theorem}[section] 
\newtheorem{proposition}[theorem]{Proposition}
\newtheorem{corollary}[theorem]{Corollary}
\newtheorem{lemma}[theorem]{Lemma}

\theoremstyle{definition} 
\newtheorem{definition}[theorem]{Definition} 
\newtheorem{problem}[theorem]{Problem}
\theoremstyle{remark} 
\newtheorem{remark}{Remark}
\newtheorem{notation}{Notation}

\newcounter{tasknumber}
\newcommand{\task}[2][]{%
  \addtocounter{tasknumber}{1}%
  \begin{center}%
  \framebox[1.1\width]{\begin{minipage}{0.9\textwidth}%
  \textbf{Task \arabic{tasknumber}} \textit{\if!#1(unassigned)!\else (#1)\fi}: {#2}%
  \end{minipage}}%
  \end{center}%
}

\newcounter{assumptionnumber}
\newcommand{\assumption}[2][]{%
  \addtocounter{assumptionnumber}{1}%
  \begin{center}%
  \framebox[1.1\width]{\begin{minipage}{0.9\textwidth}%
  \textbf{Assumption \arabic{assumptionnumber}} \textit{\if!#1!\else (#1)\fi}: {#2}%
  \end{minipage}}%
  \end{center}%
}

\newcommand{\authnote}[2][]{\noindent {\if!#1!  {\bf TODO} \else {\small \bf #1} \fi: #2}}

\newcommand{\mathsc}[1]{{\normalfont\textsc{#1}}}
\newcommand{\mathname}[1]{{\mathsc{#1}}}
\usepackage[capitalize]{cleveref}
\mathchardef\mhyphen="2D
\DeclareMathOperator{\Nrd}{Nrd}
\DeclareMathOperator{\GCD}{gcd}
\DeclareMathOperator{\Trd}{Trd}
\DeclareMathOperator{\End}{End}
\DeclareMathOperator{\length}{length}
\newcommand{\orthsum}{\obot}
\newcommand{\EndRing}{\mathname{EndRing}} 
\newcommand{\MaxOrder}{\mathname{MaxOrder}}
\newcommand{\IsoPath}{\mathname{IsogenyPath}}
\newcommand{\EllIsoPath}{{\ell\mhyphen\mathname{IsogenyPath}}}
\newcommand{\EquivIdeal}{\mathname{EquivIdeal}}
\newcommand{\QuatPath}{\mathname{QuaternionPath}}
\newcommand{\PSQuatPath}{\mathname{PS}\QuatPath}

\definecolor{NOTDISTRIBUTE}{rgb}{0.06, 0.70, 0.80} 

\begin{document}


\title[The supersingular isogeny path and endomorphism ring problems]{The supersingular isogeny path and \\endomorphism ring problems are equivalent
}
\author{Benjamin Wesolowski}
\address{Univ. Bordeaux, CNRS, Bordeaux INP, IMB, UMR 5251, F-33400, Talence, France\\
INRIA, IMB, UMR 5251, F-33400, Talence, France}
\email{benjamin.wesolowski@math.u-bordeaux.fr}
\maketitle

\begin{abstract}
We prove that the path-finding problem in $\ell$-isogeny graphs and the endomorphism ring problem for supersingular elliptic curves are equivalent under reductions of polynomial expected time, assuming the generalised Riemann hypothesis. 
The presumed hardness of these problems is foundational for isogeny-based cryptography.
As an essential tool, we develop a rigorous algorithm for the quaternion analog of the path-finding problem, building upon the heuristic method of Kohel, Lauter, Petit and Tignol. This problem, and its (previously heuristic) resolution, are both a powerful cryptanalytic tool and a building-block for cryptosystems.
\end{abstract}


\section{Introduction}

\noindent We consider two problems of foundational importance to isogeny-based cryptography, a branch of post-quantum cryptography: the endomorphism ring problem and the path-finding problem in isogeny graphs, for supersingular elliptic curves. The hardness of the first is necessary for isogeny-based cryptography to be secure~\cite{GPST16, CPV20}. Reciprocally, some cryptosystems (the earliest of which being~\cite{CGL09}) are proven secure if the second is hard. Both problems are believed to be equivalent, thereby constituting the bedrock of isogeny-based cryptography. However, known reductions rely on a variety of heuristic assumptions~\cite{PL17, EHM17, EHLMP18}. 
To arithmeticians, the endomorphism ring problem is simply the computational incarnation of the Deuring correspondence~\cite{Deuring41}. This arithmetic theory met graph theory in the work of Mestre~\cite{Mestre86} and Pizer~\cite{Pizer90}, and the related computational questions have been studied since~\cite{Kohel96}, yet the literature still heavily relies on heuristics.

This paper aims for a rigorous study of these 
problems
from the generalised Riemann hypothesis (henceforth, GRH). As tools, we develop a rigorous algorithm to solve norm equations in quaternion algebras, and a rigorous variant of the heuristic algorithm from~\cite{KLPT14} for the quaternion analog of the path-finding problem, overcoming obstacles previously deemed ``beyond the reach of existing analytic number theory techniques''~\cite{GPS20}.
As an application we prove that the path-finding problem in $\ell$-isogeny graphs and the endomorphism ring problem for supersingular elliptic curves are equivalent under reductions of polynomial expected time. 

\subsection{Hard problems for isogeny-based cryptography}

The first isogeny-based cryptosystems were proposed by Couveignes in 1997~\cite{Cou06}. This work was only made public in 2006, when the idea reemerged in~\cite{CGL09}. The latter introduced the path-finding problem in supersingular $\ell$-isogeny graphs as a possible hard problem upon which cryptosystems can be constructed.

To any primes $p$ and $\ell$ are associated a so-called \emph{supersingular $\ell$-isogeny graph}. It is  a regular graph of degree $\ell+1$ and counting approximately $p/12$ vertices. Each vertex of the graph is a \emph{supersingular elliptic curve}, and edges correspond to \emph{$\ell$-isogenies} between them (a particular kind of morphisms between elliptic curves). Most importantly, these graphs are Ramanujan, i.e., optimal expander graphs. This implies that random walks quickly reach the uniform distribution. Starting from an elliptic curve $E$, one can compute a chain of random $\ell$-isogenies until the endpoint $E'$ is uniformly distributed. Then, given only $E$ and $E'$, it seems hard to recover a path connecting them. This is the key of the preimage-resistant CGL hash function~\cite{CGL09}, and the first of our problems of interest.
\begin{problem}[$\EllIsoPath$]
Given a prime $p$, and two supersingular elliptic curves $E$ and $E'$ over $\F_{p^2}$, find a path from $E$ to $E'$ in the $\ell$-isogeny graph. 
\end{problem}

Isogeny-based cryptography has since grown considerably, when Jao and De Feo~\cite{JF11} noticed that it allows to build ``post-quantum'' cryptosystems, supposed to resist an adversary equipped with a quantum computer. There is today a wealth of other public-key protocols~\cite{CLMPR18,DGKPS19,Cos20} (including a Round 3 candidate~\cite{SIKE17} for NIST's standardisation effort), signature schemes \cite{BKV19, DFG19, GPS20, DFKLPW20} or other cryptosystems~\cite{DFMPS19, BKW20} built on the presumed hardness of finding isogenies connecting supersingular elliptic curves.

The precise relation between the security of these schemes and the supposedly hard problem $\EllIsoPath$ is a critical question. Some of these schemes, like~\cite{CGL09} or~\cite{GPS20}, are known to be secure if finding isogeny paths is hard. The reciprocal has been unclear: if one can solve $\EllIsoPath$ efficiently, is all of isogeny-based cryptography broken? The first element of response was discovered in~\cite{GPST16} by taking a detour through another problem. They prove that an efficient algorithm to solve the closely related \emph{endomorphism ring problem} allows to break the Jao--De Feo key exchange, and essentially all schemes of this type (see~\cite{FKM21}). Similarly, it was proven in~\cite{CPV20} that the security of CSIDH~\cite{CLMPR18} and its variants (an \emph{a priori} very different family of cryptosystems) also reduces to the endomorphism ring problem, via a sub-exponential reduction.

Given an elliptic curve $E$, an endomorphism is an isogeny $\varphi : E \rightarrow E$  from $E$ to itself. The set of all endomorphisms of $E$, written $\End(E)$, is a ring, where the addition is pointwise and multiplication is given by composition. Loops in $\ell$-isogeny graphs provide endomorphisms, hence the connection between path-finding problems and computing endomorphism rings. Since the curves considered are supersingular, the endomorphism rings are always generated by four elements (as a lattice), and they are isomorphic to certain subrings of a quaternion algebra $B_{p,\infty}$, called \emph{maximal orders}.
The problem of computing the endomorphism ring comes in two flavours. The first actually looks for endomorphisms.

\begin{problem}[$\EndRing$]
Given a prime $p$, and a supersingular elliptic curves $E$ over $\F_{p^2}$, find four endomorphisms of $E$ (in an efficient representation) that generate $\End(E)$ as a lattice.
\end{problem}

By an efficient representation for endomorphisms $\alpha$, we mean that there is an algorithm to evaluate $\alpha(P)$ for any $P \in E(\F_{p^k})$ in time polynomial in the length of the representation of $\alpha$ and in $k\log(p)$. We also assume that an efficient representation of $\alpha$ has length $\Omega(\log (\deg(\alpha)))$. The second version asks for an abstract description of $\End(E)$.

\begin{problem}[$\MaxOrder$]
Given a prime $p$, and a supersingular elliptic curves $E$ over $\F_{p^2}$, find four quaternions in $B_{p,\infty}$ that generate a maximal order $\cO$ such that $\cO \cong \End(E)$.
\end{problem}

Neither of them clearly reduces to the other, and in~\cite{GPST16}, it is only proven that solving both simultaneously allows to break cryptosystems. 
Many works have been studying the three problems $\EllIsoPath$, $\EndRing$ and $\MaxOrder$, as early as~\cite{Kohel96}, originally motivated by the importance of these structures in arithmetic geometry. With the increasing practical impact of these problems, it has become critical to understand their relations. It was shown in~\cite{EHLMP18} that, under several heuristic assumptions, all three appear to be equivalent.

\subsection{Contributions}
We prove that the problems $\EllIsoPath$, $\EndRing$ and $\MaxOrder$ are equivalent under reductions of polynomial expected time, assuming the generalised Riemann hypothesis. In doing so, we develop new tools for a rigorous study of these problems.

Most importantly, we develop a new, rigorous variant of the heuristic algorithm of~\cite{KLPT14} for $\QuatPath$, a quaternion analog of $\EllIsoPath$. This algorithm (and its variants) is a crucial component of the reductions, but is also a powerful cryptanalytic tool~\cite{GPST16} and a building-block for cryptosystems~\cite{DGKPS19,GPS20,DFKLPW20}. More precisely, we solve in polynomial time the following problem for very flexible choices of $\mathcal N$, including the most important variants $\ell$-$\QuatPath$ and $B$-$\PSQuatPath$.

\begin{problem}[$\QuatPath$]
Given two maximal orders $\cO_1$ and $\cO_2$ in $B_{p,\infty}$ and a set $\mathcal N$ of positive integers,
find a left $\cO_1$-ideal $I$ such that $\Nrd(I)\in \mathcal N$ and $\cO_R(I) \cong \cO_2$ (definitions provided in Section~\ref{sec:introquat}). 
If $\mathcal N$ is the set of powers of a prime $\ell$, we call the corresponding problem $\ell$-$\QuatPath$.
If $\mathcal N$ is the set of $B$-powersmooth integers for some $B>0$, we call the corresponding problem $B$-$\PSQuatPath$.
\end{problem}

The design and analysis of this new algorithm spans several sections of the present article.
\begin{itemize}
\item In Section~\ref{sec:qfprimesampling}, we combine some algorithmic considerations in euclidean lattices and the Chebotarev density theorem to prove that given an ideal in a maximal order, one can efficiently find an equivalent prime ideal (Theorem~\ref{thm:algEquivPrimeIdeal}). This serves as a preconditioning step in our algorithm, and has a heuristic analog in~\cite{KLPT14}. 
\item In Section~\ref{sec:analysis}, we prove bounds in the number of ways to represent an integer $n$ as a linear combination of a prime and a quadratic form. This is a generalisation of a classic problem of Hardy and Littlewood~\cite{HL23} on representing integers as $p + x^2 + y^2$. The proof resorts to analytic number theory, and the result, Theorem~\ref{thm:estimationofsumSi}, unlocks the analysis of  algorithms to solve certain diophantine equations in the following section.
\item In Section~\ref{sec:solving}, we design and analyse an algorithm (Theorem~\ref{thm:solvingthecomplicatedequation}) to find integral solutions $(s,t,x,y)$ of equations of the form
$$\det(\gamma)^2f(s,t) + bf^\gamma(x,y) = n,$$
where $n$ and $b$ are positive integers, $f$ is a positive definite, integral, binary quadratic form, and $\gamma$ is a $2\times 2$ integral matrix. The key allowing a rigorous analysis is to randomise the class of $f^\gamma$ within its genus using random walks, and apply the results of the previous section. As a first application, we use this algorithm to solve norm equations in special maximal orders in Corollary~\ref{coro:repintspecialorder}.
\item Finally, we piece everything together in Section~\ref{sec:newKLPT}, solving $\QuatPath$ in Theorem~\ref{thm:rigorousklpt}. The power-of-$\ell$ case is an immediate consequence, and we specialise to the powersmooth case in Theorem~\ref{thm:powersmoothclassrep}.
\end{itemize}
Note that our efforts are focused on obtaining rigorous, polynomial-time algorithms, with little consideration for practical efficiency, hence we spend little energy on calculating or optimising the hidden constants. A fast implementation should certainly follow the heuristic algorithm~\cite{KLPT14}, only resorting to our rigorous variant when unexpected obstructions are encountered. \\

This new algorithm at hands, we then tackle the various reductions between $\EllIsoPath$, $\EndRing$ and $\MaxOrder$.
They are similar to heuristic methods from the literature, and notably~\cite{EHLMP18}, with a number of substantial differences that allow a rigorous analysis. Note that our chain of reductions has a different structure from~\cite{EHLMP18}.
\begin{itemize}
\item We start in Section~\ref{sec:MOequivPath} by proving that $\EllIsoPath$ and $\MaxOrder$ are equivalent. To do so, we adapt previous heuristic methods, essentially replacing their reliance on~\cite{KLPT14} with the new rigorous variants. In particular, we prove that there is a polynomial time algorithm to convert certain ideals of prime power norm into isogenies.
\item Finally, we prove in Section~\ref{sec:MOequivEnd} that $\MaxOrder$ and $\EndRing$ are equivalent. The reduction from $\EndRing$ to $\MaxOrder$ is essentially the same as the heuristic reduction from~\cite{EHLMP18}, adapted to our new rigorous tools. The converse requires more work: the reduction from $\MaxOrder$ to $\EndRing$ in~\cite{EHLMP18} encounters several large random numbers which are hoped to be easy to factor with good probability. We propose a strategy that provably avoids hard factorisations, exploiting the tools developed in Section~\ref{sec:qfprimesampling}.
%
\end{itemize}
Note that we do not \emph{a priori} restrict the size of solutions to the three problems; however, our reductions polynomially preserve bounds on the output size. In particular, all reductions preserve the property of having a polynomially bounded output size, a requirement in~\cite{EHLMP18}. This allows the reductions to be more versatile, and apply for instance if one discovers an algorithm that solves $\EllIsoPath$ with paths of superpolynomial length.

\subsection{Notation} 
The statements $f = O(g)$, $f \ll g$ and $g = \Omega(f)$ are synonymous, where $O$ is the classic big O notation. We write $O_\varepsilon$ to signify that the hidden constants depend on $\varepsilon$. We denote by $\Z$, $\Z_{>0}$, and $\Q$ the ring of integers, the set of positive integers, and the field of rational numbers. For any prime power $q$, we denote by $\F_q$ the finite field with $q$ elements. The function $\log$ denotes the natural logarithm. The size of a set $S$ is denoted by $\#S$. If $a$ and $b$ are two integers, the greatest common divisor of $a$ and $b$ is written $\GCD(a,b)$. We write $a \mid b$ if $a$ divides $b$, or $a \mid b^\infty$ if all prime factors of $a$ divide $b$, or $a \mid\mid b$ if $a \mid b$ and $\GCD(a,b/a) = 1$. The number of divisors of $a$ is denoted by $\tau(a)$, and the number of prime divisors by $\omega(a)$, and Euler's totient is $\phi(a)$. If $R$ is a ring and $n$ a positive integer, $M_{n\times n}(R)$ is the ring of $n\times n$ matrices with coefficients in $R$. All statements containing the mention (GRH) assume the generalised Riemann hypothesis.

\section{Preliminaries}

\subsection{Quadratic forms}
We will extensively use the theory of quadratic forms; the reader can find more details on the theory, with a computational perspective, in~\cite{Cohen13}.
A quadratic form of dimension $r$ is a polynomial in $r$ variables whose terms all have degree $2$.
A quadratic form $f(x)$ in the variable $x = (x_1,\dots,x_r)$ is determined by its \emph{Gram matrix} $G = (g_{ij})$, a symmetric $r\times r$ matrix such that
$$f(x) = x^tGx = \sum_i g_{ii}x_i^2 + 2\sum_i\sum_{j>i}g_{ij}x_ix_j.$$
For computational purposes, we assume that quadratic forms are represented as their Gram matrix, and we let $\length(f)$ be the total binary length of its coefficients.
The form is \emph{integral} if $f(x) \in \Z$ for any $x \in \Z^r$, or equivalently, if $g_{ij} \in \frac 1 2 \Z$ and $g_{ii} \in \Z$. 
If $f$ is integral and $n\in \Z$, we say that $f$ represents $n$ if there exists $x\in\Z^r$ such that $f(x) = n$.
The form $f$ is \emph{definite} if $f(x) = 0$ implies $x = 0$, and it is \emph{positive} if $f(x) \geq 0$ for all $x$.
It is \emph{primitive} if the greatest common divisor of all integers represented by $f$ is $1$. It is \emph{binary} if $r = 2$.
The \emph{discriminant} of $f$ is
\begin{equation*}
\disc(f) = \begin{cases}
(-1)^{\frac r 2}\det(2G) \text{ if $r$ is even,}\\
\frac 1 2 (-1)^{\frac {r+1} 2}\det(2G) \text{ if $r$ is odd.}\\
\end{cases}
\end{equation*}
To any quadratic form $f$ is associated a symmetric bilinear form
$$\langle x,y \rangle_f = \frac{1}{2}(f(x+y) - f(x) - f(y)).$$
Given the bilinear form, one can recover the Gram matrix as $g_{ij} = \langle e_i, e_j \rangle_f$, where $(e_i)_{i=1}^r$ is the canonical basis.
If $\gamma \in M_{r\times r}(\Q)$, let $f^\gamma$ be the quadratic form defined by $f^\gamma(x) = f(\gamma x)$, with Gram matrix $\gamma^t G \gamma$.
A \emph{quadratic space} $V$ is a $\Q$-vector space of finite dimension together with a \emph{quadratic map} $q : V \rightarrow \Q$ such that for any (hence all) basis $(b_i)_{i = 1}^r$ of $V$, we have that $q\left(\sum_i x_i b_i\right)$ is a quadratic form in $x$. A \emph{lattice} is a full-rank $\Z$-submodule in a positive definite quadratic space. The discriminant of a lattice is the discriminant of the quadratic form induced by any of its bases.
Any positive definite $f$ induces a lattice structure on $\Z^r$, via the canonical basis. The geometric invariants of this lattice induce invariants of $f$.
The \emph{volume} of $f$ is $\Vol(f) = |\det(G)|^{1/2}$. The \emph{covering radius} $\mu(f)$ is the smallest $\mu$ such that for any $y\in\R^r$, we have $\min_{x\in\Z^r}f(x-y) \leq \mu^2$. We will use the following bound.

\begin{lemma}\label{lem:boundoncoveringradius}
If $f$ is integral, then $\mu(f) \leq \frac{1}{2}r^{1/2}\gamma_r^{ r/2}\Vol(f)$, where $\gamma_r$ is Hermite's constant.
\end{lemma}
\begin{proof}
Let $\lambda_i$ be the successive minima of $f$.
We have $\frac 1 2 \lambda_r \leq \mu(f) \leq \frac{r^{1/2}}{2} \lambda_r$.
By Minkowski's second theorem,
$\prod_{i = 1}^r\lambda_i \leq \gamma_r^{ r/2}\Vol(f),$
and since $f$ is integral, $\lambda_i \geq 1$, hence
$\lambda_r \leq \gamma_r^{ r/2}\Vol(f).$
\end{proof}


\subsection{Quaternion algebras}\label{sec:introquat}
An algebra $B$ is a \emph{quaternion algebra over $\Q$} if there exist $a,b\in \Q^\times$ and $i,j\in B$ such that $(1,i,j,ij)$ is a $\Q$-basis for $B$ and
$$i^2 = a,\ \ \ j^2 = b,\ \ \ \text{and }ji = -ij.$$
Given $a$ and $b$, the corresponding algebra is denoted by $\left(\frac{a,b}{\Q}\right)$.
Write an arbitrary element of $B$ as $\alpha = x_1 + x_2 i + x_3 j + x_4 ij$ with $x_i \in \Q$.
The quaternion algebra $B$ has a canonical involution $\alpha \mapsto \overline{\alpha} = x_1 - x_2 i - x_3 j - x_4 ij$.
It induces the \emph{reduced trace}
and the \emph{reduced norm}
$$\Trd(\alpha) = \alpha + \overline{\alpha} = 2x_1,\ \ \ \Nrd(\alpha) = \alpha \overline{\alpha} = x_1^2 - ax_2^2 - bx_3^2 + abx_4^2.$$
The latter is a quadratic map, which makes $B$ a quadratic space, and endows its $\Z$-submodules with a lattice structure. The corresponding bilinear form is
$$\langle \alpha, \beta\rangle = \frac{1}{2}\left(\alpha\overline{\beta} + {\beta}\overline{\alpha}\right).$$
If $\Lambda$ is a full-rank lattice in $B$, the \emph{reduced norm} of $\Lambda$ is $\Nrd(\Lambda) = \gcd\left( \Nrd(\alpha) \mid \alpha \in \Lambda \right)$. We associate to $\Lambda$ the normalised quadratic map
$$q_\Lambda : \Lambda \longrightarrow \Z : \lambda \longmapsto \frac{\Nrd(\lambda)}{\Nrd(\Lambda)}.$$
An \emph{order} $\cO$ in $B$ is a full-rank lattice that is also a subring. It is \emph{maximal} if it is not contained in any other order. For any lattice $\Lambda \subset B$, we define the \emph{left order of $\Lambda$} and the \emph{right order of $\Lambda$} as
\begin{align*}
\cO_L(\Lambda) &= \left\{\alpha \in B \mid \alpha \Lambda  \subseteq \Lambda \right\},\ \ \ \text{and }
\cO_R(\Lambda) = \left\{\alpha \in B \mid \Lambda \alpha \subseteq \Lambda \right\}.
\end{align*} 
If $\cO$ is a maximal order, and $I$ is a left ideal in $\cO$, then $\cO_L(I) = \cO$ and $\cO_R(I)$ is another maximal order. Given two maximal orders $\cO_1$ and $\cO_2$, their \emph{connecting ideal} is the ideal
$$I(\cO_1,\cO_2) = \{\alpha \in B \mid \alpha \cO_2 \overline \alpha \subseteq [\cO_2:\cO_1\cap \cO_2]\cO_1\},$$
which satisfies $\cO_L(I) = \cO_1$ and $\cO_R(I) = \cO_2$.

Let $\cO$ be a maximal order. Two left $\cO$-ideals $I$ and $J$ are \emph{equivalent} if there exists $\alpha \in B$ such that $I = \alpha J$. The set of classes for this equivalence relation is the \emph{(left) ideal class set} of $\cO$, written $\mathrm{Cls}(\cO)$. The class of $I$ is written $[I]$.
\\

To any prime number $p$, one associates a quaternion algebra $B_{p,\infty}$. In algebraic terms, $B_{p,\infty}$ is defined as the unique quaternion algebra over $\Q$ ramified exactly at $p$ and $\infty$. Explicitly, it is given by the following lemma, from~\cite{Pizer80}.
\begin{lemma}\label{lem:quaternionpq}
Let $p > 2$ be a prime. Then, $B_{p,\infty} = \left(\frac{-q,-p}{\Q}\right)$, where
$$q = \begin{cases}
1&\text{if } p \equiv 3 \bmod 4,\\
2&\text{if } p \equiv 5 \bmod 8,\\
q_p&\text{if } p \equiv 1 \bmod 8,\\
\end{cases}$$
where $q_p$ is the smallest prime such that $q_p \equiv 3 \bmod 4$ and $\left(\frac p {q_p}\right) = -1$.
Assuming GRH, we have $q_p = O((\log p)^2)$, which can thus be computed in polynomial time in $\log p$.
\end{lemma}

For a given quaternion algebra, the defining pair $(a,b)$ is not unique. However, in the rest of this article, the algebra $B_{p,\infty}$ will always be associated to the pair $(-q,-p)$ given in Lemma~\ref{lem:quaternionpq}, and the induced basis $(1,i,j,ij)$.
For each $p$, we distinguish a maximal order $\cO_0$ in $B_{p,\infty}$, and a useful suborder $R+Rj$ in the following lemma. This order $\cO_0$ will be reffered to as the \emph{special maximal order} of $B_{p,\infty}$.

\begin{lemma}\label{lem:specialorders}
For any $p>2$, the quaternion algebra $B_{p,\infty}$ contains the maximal order 
$$\cO_0 = \begin{cases}
\left\langle 1, i, \frac{i+ij}{2}, \frac {1+j}{2}\right\rangle &\text{if } p \equiv 3 \bmod 4,\\
\left\langle 1, i, \frac{2-i+ij}{4}, \frac {-1+i+j}{2}\right\rangle &\text{if } p \equiv 5 \bmod 8,\\
\left\langle \frac{1+i}{2}, \frac{j+ij}{2}, \frac {i+cij}{q}, ij\right\rangle &\text{if } p \equiv 1 \bmod 8,\\
\end{cases}$$
where in the last case $c$ is an integer such that $q \mid c^2p+1$. Assuming GRH, the maximal order $\cO_0$ contains the suborder 
$R + Rj$
with index $O((\log p)^2)$, where $R$ is the ring of integers of ${\Q(i)}$. If $\omega$ is a reduced generator of $R$, then
$$\Nrd(s+t\omega + xj +y\omega j) = f(s,t) + p f(x,y),$$
where $f$ is a principal, primitive, positive definite, integral binary quadratic form of discriminant $\disc(\Q(i)) = O((\log p)^2)$.
\end{lemma}
\begin{proof}
This lemma summarises~\cite[Section~2.2]{KLPT14}, itself based on~\cite{Pizer80} and~\cite{LO77}.
\end{proof}

If $\cO$ is any maximal order in $B_{p,\infty}$, then $\disc(\cO) = p^2$. In fact, for any left $\cO$-ideal $I$, we have $\#(\cO/I) = \Nrd(I)^2$ and the normalised quadratic map $q_I$ has discriminant $p^2$. The following lemma tells us that the integers represented by $q_I$ are the norms of ideals equivalent to $I$.

\begin{lemma}[\protect{\cite[Lemma~5]{KLPT14}}]
Let $I$ be a left $\cO$-ideal, and $\alpha \in I$. Then, $I\overline{\alpha}/\Nrd(I)$ is an equivalent left $\cO$-ideal of norm $q_I(\alpha)$.
\end{lemma}

\subsection{Supersingular elliptic curves}
A detailed account of the theory of elliptic curves can be found in~\cite{Silverman-Arithmetic}.
An \emph{elliptic curve} is an abelian variety of dimension $1$.
More explicitly, given a field $k$ of characteristic $p>3$, an elliptic curve $E$ can be described as an equation $y^2 = x^3 + Ax + B$ for $A,B \in k$ with $4A^3 + 27B^2 \neq 0$. The \emph{$k$-rational points} of $E$ is the set $E(k)$ of pairs $(x,y) \in k^2$ satisfying the curve equation, together with a point $\infty_E$ `at infinity'. They form an abelian group, written additively, where $\infty_E$ is the neutral element. The \emph{geometric points} of $E$ are the $\overline k$-rational points, where $\overline k$ is the algebraic closure of $k$.

Let $E_1$ and $E_2$ be two elliptic curves defined over $k$. An \emph{isogeny} $\varphi : E_1 \rightarrow E_2$ is a non-constant rational map that sends $\infty_{E_1}$ to $\infty_{E_2}$. It is then a group homomorphism from $E_1(\overline k)$ to $E_2(\overline k)$, and is its kernel over the algebraic closure, written  $\ker(\varphi)$, is finite. The \emph{degree} $\deg(\varphi)$ is the degree of $\varphi$ as a rational map. When $\deg(\varphi)$ is coprime to $p$, then $\deg(\varphi) = \#\ker(\varphi)$.
The degree is multiplicative, in the sense that $\deg(\psi\circ\varphi) = \deg(\psi)\deg(\varphi)$.
For any integer $n \neq 0$, the multiplication-by-$m$ map $[m] : E \rightarrow E$ is an isogeny.
For any isogeny $\varphi : E_1 \rightarrow E_2$, its \emph{dual} is the unique isogeny $\hat \varphi : E_2 \rightarrow E_1$ such that $\hat\varphi\circ \varphi =[\deg(\varphi)]$.
If $\deg(\varphi) = \ell$ is prime, we say that $\varphi$ is an $\ell$-isogeny. 
Any isogeny factors as a product of isogenies of prime degrees, hence $\ell$-isogenies are basic building blocks. An isogeny of degree coprime to $p$ is uniquely determined by its kernel. Given this kernel, one can compute equations for the isogeny is time polynomial in $\deg(\varphi)$ and $\log p$ via V\'elu's formula~\cite{velu}. An isogeny can be represented in size polynomial in $\log p$ and $\deg(\varphi)$, for instance as a rational map, or by a generator of its kernel. The output of $\EllIsoPath$ is a chain of $\ell$-isogenies of length $k$; it corresponds to an isogeny of degree $\ell^k$, but should be represented as a sequence of $\ell$-isogenies (so that the length of the representation is polynomial in $\ell$ and $k$ instead of $\ell^k$).

An \emph{isomorphism} is an isogeny $\varphi : E_1 \rightarrow E_2$ of degree $1$. We say that $E_1$ and $E_2$ are isomorphic over $K$ (an extension of $k$) if there is an isomorphism between them that is defined over $K$. The $j$-invariant of $E$ is $j(E) = \frac{256\cdot27\cdot A^3}{4A^3 + 27B^2}$. We have $j(E_1) = j(E_2)$ if and only if $E_1$ and $E_2$ are isomorphic over the algebraic closure of $k$. It is then simple to test $\overline k$-isomorphism. It is also simple to compute explicit isomorphisms.

An endomorphism of $E$ is an isogeny $\varphi : E \rightarrow E$ from $E$ to itself. The endomorphism ring $\End(E)$ is the collection of these endomorphisms, together with the trivial map $\varphi(x,y) = \infty_E$. It is a ring for pointwise addition, and for composition of maps. The map $\Z \rightarrow \End(E) : m \mapsto [m]$ is an embedding. In that sense, $\End(E)$ contains $\Z$ as a subring, but it is always larger (in positive characterisic). The curve $E$ is \emph{supersingular} if $\End(E)$ has rank $4$ as a $\Z$-module. Then, $\End(E)$ is isomorphic to a maximal order in the quaternion algebra $B_{p,\infty}$, defined in Section~\ref{sec:introquat}.
Up to $\overline k$-isomorphism, all supersingular elliptic curves are defined over $\F_{p^2}$, and there are $\lfloor p/12\rfloor + \varepsilon$ of them, with $\varepsilon \in \{0,1,2\}$. 
Fix a prime $\ell \neq p$. The supersingular $\ell$-isogeny graph (for $p$) is the graph whose vertices are these supersingular elliptic curves (up to isomorphism), and there is an edge from $E_1$ to $E_2$ for each $\ell$-isogeny from $E_1$ to $E_2$. It is a regular graph of degree $\ell+1$ (because any $E$ has $\ell+1$ subgroups $H$ of order $\ell$, each inducing an isogeny of kernel $H$). The $\ell$-isogeny graph is Ramanujan. In particular, random walks rapidly converge to the uniform distribution, and any two curves of the graph are connected by an isogeny of degree $\ell^{m}$ with $m = O(\log p)$. 

\subsection{The Deuring correspondence}
As already mentioned, given a supersingular elliptic curve $E$ over $\F_{p^2}$, its endomorphism ring $\End(E)$ is isomorphic to a maximal order in $B_{p,\infty}$.
This \emph{Deuring correspondence} is in fact a bijection
$$\left\{\substack{\text{Isomorphism classes of}\\ \text{maximal orders $\cO$ in $B_{p,\infty}$}}\right\} \longleftrightarrow \left\{\substack{\text{Isomorphism classes of }\\ \text{supersingular elliptic curves $E$ over $\F_{p^2}$}}\right\}/\Gal(\F_{p^2}/\F_p).$$
A more detailed account of the theory can be found in~\cite[Chapter~42]{VoightBook}.
\\

We have identified a special order $\cO_0$ in Lemma~\ref{lem:specialorders}, and it is natural to wonder what the corresponding elliptic curve may be. If $p \equiv 3 \bmod 4$, then the curve $E_0$ defined by $y^2 = x^3 - x$ is supersingular. It is defined over $\F_p$, so has the Frobenius endomorphism $\pi : (x,y) \mapsto (x^p,y^p)$. Furthermore, if $\alpha \in \F_{p^2}$ satisfies $\alpha^2 = -1$, it is easy to check that $\iota : (x,y) \mapsto (-x,\alpha y)$ is also an endomorphism. These endomorphisms generate almost all $\End(E_0)$: we actually have
$$\End(E_0) = \Z \oplus \Z \iota \oplus \Z \frac{\iota+\iota\pi}{2} \oplus \Z \frac{1+\pi}{2}.$$
Since $\iota^2 = [-1]$ and $\pi^2 = [-p]$, we have $\End(E_0) \cong \cO_0$. More generally, we have the following result.

\begin{lemma}[\protect{\cite[Proposition~3]{EHLMP18}}]\label{lem:constructingE0}
Let $\cO_0$ as in Lemma~\ref{lem:specialorders}.
There is an algorithm that for any prime $p > 2$ computes an elliptic curve $E_0$  over $\F_p$ and $\iota\in \End(E_0)$ such that
$$\cO_0 \longrightarrow \End(E_0) : 1,i,j,ij \longmapsto [1], \iota, \pi, \iota \pi$$
is an isomorphism, and runs in time polynomial in $\log p$ (if $p \equiv 1 \bmod 8$, we assume GRH).
\end{lemma}

The Deuring correspondence runs deeper than a simple bijection: it also preserves morphisms between the two categories. Given any isogeny $\varphi : E_1 \rightarrow E_2$, let $I_\varphi = \Hom(E_2,E_1)\varphi$, where $\Hom(E_2,E_1)$ is the set of isogenies from $E_2$ to $E_1$. This object $I_\varphi$ is a left $\End(E_1)$-ideal, hence $\cO_L(I_\varphi) \cong \End(E_1)$. Furthermore, $\cO_R(I_\varphi) \cong \End(E_2)$. In other words, $I_\varphi$ connects $\End(E_1)$ to $\End(E_2)$, just as $\varphi$ connects $E_1$ to $E_2$. This construction preserves the `quadratic structure', in the sense that $\Nrd(I_\varphi) = \deg(\varphi)$.

Conversely, suppose $I$ is a left $\End(E_1)$-ideal. Then, we can construct an isogeny $\varphi_I$ as the unique isogeny with kernel
$\bigcap_{\alpha \in I} \ker(\alpha).$
These two constructions are mutual inverses, meaning that for any $I$ and $\varphi$, we have $I_{\varphi_I} = I$ and $\varphi_{I_\varphi} = \varphi$.
The translation from $I$ to $\varphi_I$ can be computed efficiently, provided that $I$ is an ideal in the special order $\cO_0$ from Lemma~\ref{lem:specialorders}, and that $\Nrd(I)$ is powersmooth (its prime-power factors are polynomially bounded). This is the following lemma. Only the case $p \equiv 3 \bmod 4$ is considered in~\cite{GPS20}, but as noted in~\cite{EHLMP18}, it easily extends to arbitrary $p$.

\begin{lemma}[\protect{\cite[Lemma~5]{GPS20}}]\label{lem:idealtoisogeny}
Let $\cO_0$ as in Lemma~\ref{lem:specialorders}, and $E_0$ as in Lemma~\ref{lem:constructingE0}.
There exists an algorithm which, given a left $\cO_0$-ideal $I$ of norm $N = \prod_i \ell_i^{e_i}$, returns the corresponding isogeny $\varphi_I : E_0 \rightarrow E_1$. The complexity of this algorithm is polynomial in $\log p$ and $\max_i (\ell_i^{e_i})$ (if $p \equiv 1 \bmod 8$, we assume GRH).
\end{lemma}


\section{Quadratic forms and prime sampling}\label{sec:qfprimesampling}

\noindent
In this section, we consider the following problem: given an integral, primitive, positive definite quadratic form $f$ of rank $r$, find $x\in\Z^r$ such that $f(x)$ is prime. We then give a first application of this problem, for finding ideals of prime norm in a given ideal class of a maximal order of $B_{p,\infty}$.

\subsection{Sampling primes} Let $f$ be an integral, primitive, positive definite quadratic form. In this section, we discuss the problem of sampling vectors in $\{x \in \Z^r \mid f(x) \leq \rho\}$ so that $f(x)$ is prime. Let us first focus on the binary case, for which the following theorem tells that an important proportion of vectors represent primes.
It is a classical consequence of the effective Chebotarev density theorem under GRH, due to Lagarias and Odlyzko~\cite{LO77}.

\begin{theorem}[GRH]\label{thm:primebinaryform}
If $f$ is an integral, primitive, positive definite, binary quadratic form of discriminant $D$, the number of primes at most $\rho$ represented by $f$ is
$$\pi_f(\rho) = \frac{\delta \rho}{h(D)\log \rho} + O(\rho^{1/2}\log(|D|\rho)),$$
where $\delta$ is $1$ is $f(x,y)$ is equivalent to $f(x,-y)$, and $1/2$ otherwise.
\end{theorem}


The quantity $\pi_f(\rho)$ should be compared to the cardinality of $\{(x,y) \in \Z^2 \mid f(x,y) \leq \rho\}$, which we estimate in the following lemma, in a slightly more general form for later purposes.

%

\begin{lemma}\label{lem:balldim2}
For any integral, positive definite, binary quadratic form $f$, any $x_0 \in \R^2$ and any $\rho \geq 0$, we have
$$\left|\#\{x \in \Z^2 \mid f(x+x_0) \leq \rho\} - \frac{\pi}{\Vol(f)}\rho\right| \leq   2\pi\sqrt{\frac 2 3} \rho^{1/2} +\frac {2\pi}{3} \Vol(f).$$
\end{lemma}
\begin{proof}
For any $z \geq 0$, let $V_2(z) = \pi z^2$ be the volume of the standard 2-ball of radius $z$. It is a classical application of the covering radius $\mu(f)$ that 
$$V_2(\rho^{1/2} - \mu(f)) \leq \Vol(f) \cdot \#\{x \in \Z^2 \mid f(x+x_0) \leq \rho\} \leq V_2(\rho^{1/2} + \mu(f)).$$
This comes from the fact that Voronoi cells of $f$ have volume $\Vol(f)$ and diameter $2\mu(f)$.
From Lemma~\ref{lem:boundoncoveringradius} with Hermite's constant $\gamma_2 = 2/\sqrt{3}$, we have $\mu(f) \leq \sqrt{\frac 2 3} \Vol(f)$. We obtain
$$\left|\left(\rho^{1/2} \pm \mu(f)\right)^2 -\rho \right| \leq   2\mu(f)\rho^{1/2} +\mu(f)^2 \leq 2\sqrt{\frac 2 3} \Vol(f)\rho^{1/2} +\frac 2 3 \Vol(f)^2,$$
from which the result follows.
\end{proof}

\begin{lemma}\label{lem:samplingbinary}
Let $f$ be a primitive, positive definite, integral, binary quadratic form, and let $\rho > 0$. There is an algorithm that samples uniformly random elements from $$\left\{(x,y) \in \Z^2 \mid f(x,y) \leq \rho\right\}$$ in polynomial time in $\log \rho$ and in $\length(f)$.
\end{lemma}

\begin{proof}
Let $B_f(r) = \{v \in \R^2 \mid f(v) \leq r^2\}$ be the ball of radius $r$ around the origin.
Let $r = \rho^{1/2}$, and we wish to sample uniformly in $B_f(r) \cap \Z^2$.
First, compute a Minkowski-reduced basis $(b_1,b_2)$ of $f$ with $f(b_1) \leq f(b_2)$.
If $\rho < f(b_2)$, then $B_f(r) \cap \Z^2 \subset \Z b_1$, and we can uniformly sample $k \in \Z$ such that $k^2 \leq \rho/f(b_1)$ and return $kb_1$.
We may now assume that $\rho \geq f(b_2)$, which implies $r \geq \sqrt{2}\mu$, with $\mu$ the covering radius of $f$.
Let $\mathscr V = \{v \in \R^2 \mid f(v) = \min_{\lambda \in \Z^2}f(v+\lambda)\}$ be the Voronoi cell around the origin. Given any $v \in \R^2$, a closest lattice vector is an element $\lambda(v) \in \Z^2$ such that $v \in \mathscr V + \lambda(v)$. This closest vector can be computed efficiently in dimension $2$, and is unique for almost all $v$: only the boundaries of Voronoi cells are ambiguous. We sample as follows:
\begin{enumerate}
\item Sample $v \in B_f(r + \mu)$ uniformly.
\item Solve the closest vector problem for $v$, resulting in $\lambda(v)$ (unique with probability $1$).
\item If $\lambda(v) \in B_f(r)$, return it; otherwise restart.
\end{enumerate}
Let us analyse the distribution of $\lambda(v)$ when $v \in B_f(r + \mu)$ is uniform. For any $u \in \Z^2 \cap B_f(r)$, we have $\mathscr V+u \subset B_f(r+\mu)$, hence
$$\Pr[\lambda(v) = u] = \frac{\Vol((\mathscr V+u) \cap B_f(r+\mu))}{\Vol (B_f(r+\mu)))} = \frac{\Vol(\mathscr V)}{\Vol (B_f(r+\mu)))}.$$
In particular, $\Pr[\lambda(v) = u]$ for $u \in \Z^2 \cap B_f(r)$ does not depend on $u$, which proves that the output of the sampling procedure is uniform in $\Z^2 \cap B_f(r)$. Finally,
\begin{align*}
\Pr[\lambda(v) \in B_f(r)] &\geq \Pr[v \in B_f(r-\mu)]
= \frac{(r-\mu)^2}{(r+\mu)^2}
\geq \left(\frac{1-2^{-1/2}}{1+2^{-1/2}}\right)^2 \geq 0.028,
\end{align*}
which proves that the procedure succeeds after an expected constant number of trials.
\end{proof}

\begin{proposition}[GRH]\label{prop:samplingprimebinary}
Let $f$ be a primitive, positive definite, integral, binary quadratic form. For any $\varepsilon>0$, there is an algorithm that finds integers $x,y \in \Z$ such that $f(x,y)$ is a prime number at most $O_\varepsilon(|\disc(f)|^{1+\varepsilon})$, and runs in polynomial time in $\length(f)$.
\end{proposition}

\begin{proof}
From Lemma~\ref{lem:samplingbinary}, one can sample uniformly random pairs of integers $(x,y) \in \Z^2$ such that $f(x,y) \leq \rho$. We conclude by combining Lemma~\ref{lem:balldim2} and Theorem~\ref{thm:primebinaryform} (with $h(D) = O( |D|^{1/2} \log |D|)$, see for instance \cite[p.~138]{Cohen08}), which imply that a uniformly random vector represents a prime with good probability. The $\varepsilon$ in the exponent comes from the crossover point between the main term and the error term in Theorem~\ref{thm:primebinaryform}.
\end{proof}

\begin{proposition}[GRH]\label{prop:samplingprimegeneral}
Let $f$ be a primitive, positive definite, integral quadratic form of dimension~$r\geq 3$. For any $\varepsilon > 0$, there is an algorithm that finds a vector $x \in \Z^r$ such that $f(x)$ is a prime number at most $O_\varepsilon\left(\left(2^{r(r-1)}|\disc(f)|\right)^{1+\varepsilon}\right)$ (or $O_\varepsilon\left(|\disc(f)|^{4/3+\varepsilon}\right)$ if $r = 3$), and runs in polynomial time in $\length(f)$.
\end{proposition}

\begin{proof}
We are looking for two integral vectors $u$ and $v$ that generate a primitive binary quadratic form $g_{uv}(x,y) = f(xu+yv)$. We then apply Proposition~\ref{prop:samplingprimebinary} to $g_{uv}$.

Compute an LLL-reduced~\cite{lenstra82:_factor} basis $(b_1,b_2,\dots,b_r)$ of $f$ so that $f(b_1) \leq 2^{r-1}\Vol(f)^{2/r}$ and $\prod_i f(b_i) \leq 2^{r(r-1)}\Vol(f)^2$. 
Let $u = b_1$. Then, factor $f(u) = \prod_i a_i^{e_i}$ where each $a_i$ seems hard to factor further, and they are pairwise coprime. For each $a_i$, we now describe a procedure that will either reveal new factors of $a_i$ (in which case we can restart with this new piece of information), or find a vector $v_i$ such that 
$f(v_i) = \langle v_i,v_i \rangle$ is coprime to $a_i$. We proceed as follows:
\begin{enumerate}
\item We compute the greatest common divisor of $a_i$ with each of $\langle b_j,b_j \rangle$ and $2\langle b_j,b_k \rangle$ (i.e., the coefficients of $f$ in the basis $b_1,\dots,b_r$).
\item These common divisors cannot all be equal to $a_i$ since $f$ is primitive. So either one of them is a non-trivial factor of $a_i$ (and we restart), or one of them is $1$.
\item If there is an index $j$ such that $\GCD(\langle b_j,b_j \rangle, a_i) = 1$, we return $v_i = b_j$.
\item Otherwise, there are indices $j$ and $k$ with $\GCD(\langle b_j,b_j \rangle, a_i) = \GCD(\langle b_k,b_k \rangle, a_i) = a_i$, and $\GCD(2\langle b_j,b_k \rangle, a_i) = 1$. Then, we return $v_i = b_j+b_k$.
\end{enumerate}
Now, let $v = \sum_{i}v_i\prod_{j\neq i}a_j$, and as desired, $\GCD(f(u),f(v)) = 1$. We have $f(v) \ll f(b_1)^2 f(b_r)$, and the form $g_{uv}$ is primitive.
It has volume at most $\sqrt{f(u)f(v)}$. If $r>3$, we have $\sqrt{f(u)f(v)}\ll\sqrt{f(b_1)f(b_2)f(b_3)f(b_r)} \leq 2^{r(r-1)/2}\Vol(f)$, and if $r = 3$, we have $\sqrt{f(u)f(v)}\ll \Vol(f)^{4/3}$. The result then follows from Proposition~\ref{prop:samplingprimebinary}.
\end{proof}
%

In applications, we will often need to find vectors representing primes that are \emph{large enough} (but not too large). This can be done in a straighforward adaptation of the above strategy.

\begin{proposition}[GRH]\label{prop:samplingprimegenerallarge}
There exists a constant $c$ and an algorithm $\mathscr A$ such that the following holds. 
Let $f$ be a primitive, positive definite, integral quadratic form of dimension~$r$. For any $\rho > (2^{r^2}|\disc(f)|)^c$, the algorithm $\mathscr A(f,\rho)$ outputs a vector $x \in \Z^r$ such that $f(x)$ is a prime number between $\rho$ and $\rho^2$, and runs in polynomial time in $\length(f)$ and $\log \rho$.
\end{proposition}

\begin{proof}
As in the proof of Proposition~\ref{prop:samplingprimegeneral}, $\mathscr A$ can compute a sub-basis of $f$ that induces a primitive binary quadratic form $g$ of discriminant at most $O_\varepsilon\left(\left(2^{r(r-1)}|\disc(f)|\right)^{1+\varepsilon}\right)$.
Applying Lemma~\ref{lem:samplingbinary}, one can sample uniformly random pairs $(x,y)$ such that $g(x,y) \leq \rho^2$. From Theorem~\ref{thm:primebinaryform}, $g(x,y)$ is prime and larger that $\rho$ with good probability, provided that $\rho$ is large enough.
\end{proof}

\subsection{Computing equivalent ideals of prime norm}
The above results will be important in the rest of the article, and we can already prove them useful with a first important application.
Consider a maximal order $\cO$ in $B_{p,\infty}$ and a left $\cO$-ideal $I$. We can compute an equivalent ideal $J$ of prime norm as an immediate consequence of Proposition~\ref{prop:samplingprimegeneral}.

\begin{algorithm}[h]
 \caption{$\mathname{EquivPrimeIdeal}_\varepsilon(I)$}\label{alg:EquivPrimeIdeal}
 \begin{algorithmic}[1]
 \REQUIRE {A left ideal $I$ in a maximal order $\cO$.}
 \ENSURE {An ideal $J$ of prime norm, and an element $\alpha \in I$ such that $J = I\overline{\alpha}/\Nrd(I)$.}

\STATE $\alpha \gets $ an element $\alpha \in I$ such that $q_I(\alpha)$ is prime; \COMMENT{Proposition~\ref{prop:samplingprimegeneral}}
\RETURN $J = I\overline{\alpha}/\Nrd(I)$, and $\alpha$.
  \end{algorithmic}
 \end{algorithm}
 
 \begin{theorem}[GRH]\label{thm:algEquivPrimeIdeal}
For any $\varepsilon>0$, Algorithm~\ref{alg:EquivPrimeIdeal} is correct and runs in expected polynomial time in $\log \Nrd(I)$ and~$\log p$, and the output $J$ has reduced norm $\Nrd(J) = O_\varepsilon (p^{2+\varepsilon})$.
 \end{theorem}
 
\begin{proof}
It follows from Proposition~\ref{prop:samplingprimegeneral}, and the fact that $q_I : I \rightarrow \Z$ is a primitive, positive-definite, integral quadratic map of discriminant $p^2$. 
\end{proof}

\begin{remark}
Recall that our efforts are focused on provability, and the constants we obtain are certainly not tight. In~\cite[Section 3.1]{KLPT14}, the analogue heuristic algorithm is  expected to return $J$ of norm $\Nrd(J) = \tilde O (p^{1/2})$ most of the time, and they argue that in the worst case, one could possibly obtain $\Nrd(J) = \tilde O(p)$.
\end{remark}


For our applications, we need a slightly more powerful version.
 \begin{proposition}[GRH]\label{prop:algEquivPrimeIdealLargeNonQuad}
There is a constant $c$ and an algorithm which on input a left ideal $I$, a bound $\rho > p^c$, and a prime $\ell \neq p$, returns an ideal $J$ equivalent to $I$ such that $\Nrd(J)$ is a prime between $\rho$ and $\rho^2$, and $\ell$ is a non-quadratic residue modulo $\Nrd(J)$, and runs in polynomial time in $\log \Nrd(I)$,~$\log p$, and $\ell$.
 \end{proposition}
 \begin{proof}
Apply Algorithm~\ref{alg:EquivPrimeIdeal} with two modifications. First, we use Proposition~\ref{prop:samplingprimegenerallarge} instead of Proposition~\ref{prop:samplingprimegeneral}. Second, assuming $\ell \neq 4$ we consider a sublattice $4\ell I \subset \Lambda \subset I$ in place of $I$, where
the quotient $\Lambda/4\ell I$ is generated by any element $x$ such that $q_I(x) \equiv 1 \bmod 4$ and $q_I(x)$ is a non-quadratic residue modulo $\ell$. It follows from quadratic reciprocity that for any $y$ in the lattice, when $q_I(y)$ is prime, then $\ell$ is a non-quadratic residue modulo $q_I(y)$. Similarly, if $\ell = 2$, we consider a sublattice $8 I \subset \Lambda \subset I$ where the quotient $\Lambda/8 I$ is generated by any element $x$ such that $q_I(x) \equiv 3 \text{ or } 5 \bmod 8$.
 \end{proof}

\section{Representing integers with quadratic forms and primes}\label{sec:analysis}

\noindent
In this section, we count the number of ways to represent an integer $n$ in the form $a\ell + bf(x,y)$, where the integers $a$ and $b$ and the quadratic form $f$ are fixed, and $\ell$ is required to be prime. The bounds we obtain are key to the analysis of algorithms designed in the following sections. The proof resorts to analytic number theory.  The reader only interested in computational applications can safely read up to Corollary~\ref{coro:coromaintheoremthatimpliesklpt} before skipping to the next section.\\

We fix the following notation for the rest of the section. Let $f$ be a primitive, integral, positive definite, binary quadratic form of discriminant $f_\chi f_0^2$ where $f_\chi$ is fundamental. Let $v = f_\chi f_0$.
Let $a,b,u$ be positive integers with $\GCD(a,b) = \GCD(b,v) = \GCD(u,v) = \GCD(a,f_\chi) = 1$. 
Let $\chi$ be the Kronecker symbol $\chi(m) = \left(\frac{f_\chi}{m}\right)$, primitive of conductor $f_\chi$.
Let $n$ be a positive integer such that $\GCD(a,n) = \GCD(b,n) =\GCD(n-au, f_0) = 1$. 
Finally, let
$$\mathscr S_n(f) = \{(\ell,x,y) \mid a\ell + bf(x,y) = n, \text{ where }x,y,\ell \in \Z, \ell \text{ is prime, and } \ell \equiv u \bmod v\}.$$
The goal of this section is to obtain lower bounds on the size of $\mathscr S_n(f)$.\\

The problem at hand is a generalisation of the classic problem of Hardy and Littlewood~\cite{HL23} of representing integers as $\ell + x^2 + y^2$, where $\ell$ is prime.
The \emph{number of representations} of an integer $N$ by $f$ is
$r(N;f) = \#\{(x,y) \in \Z^2 \mid f(x,y) = N\}.$
Following a classical approach to the Hardy and Littlewood problem, we can write
\begin{align*}
\#\mathscr S_n(f) &= \sum_{\substack{\ell \leq n/a\\ \ell\text{ prime}\\ \ell \equiv n/a \bmod b\\ \ell \equiv u \bmod v}} r\left(\frac{n-a\ell}{b};f\right).
\end{align*}
Unfortunately, controlling $r(N,f)$ is in general a difficult task. However, we know more about the number of representations of $N$ in the genus of $f$. We indeed have the following classical theorem (see for instance~\cite{Pall33}).


\begin{theorem}\label{thm:sumkronecker}
Let $f$ be a primitive, integral, binary quadratic form of discriminant $D = dm^2$, where $d$ is fundamental.
For any $N>0$ such that $\GCD(N,m)=1$, the number of representations of $N$ by forms of discriminant $D$ is
$$w\sum_{\nu\mid N} \chi(\nu),$$
where $\chi(\nu) = \left(\frac d \nu \right)_K$ is the Kronecker symbol, and $w = 4$ when $d = 4$, $w = 6$ when $d = 3$, and $w = 2$ otherwise.
\end{theorem}

Let $(f_i)_{i=1}^t$ be a list of class representatives for each form of same discriminant as $f$. The main result of this section is the following theorem.
\begin{theorem}[GRH]\label{thm:estimationofsumSi}
There exists an absolute constant $\delta > 0$ such that for any integer $n \geq \max(a,b,v)^{1/\delta}$, we have
\begin{equation*}
\left|\sum_{i=1}^t\#\mathscr S_n(f_i) - \frac{n}{a\log(n/a)}\frac{w}{\phi(b)\phi(v)}\left(1 + \chi\left(\frac{n - ua}{b}\right)\right)L(1,\chi)C(\chi,anf_0,b) \right| = O(n^{1-\delta}),
\end{equation*}
where $L(s,\chi)$ is the Dirichlet $L$-function, the integer $w$ is as in Theorem~\ref{thm:sumkronecker}, and
\begin{align*}
C(\chi,m,s) &= \prod_{\ell\nmid ms}\left(1 + \frac{\chi(\ell)}{\ell(\ell-1)}\right)\prod_{\ell \mid m }\left(1 - \frac{\chi(\ell)}{\ell}\right).
\end{align*}
\end{theorem}

The following (immediate) corollary is more convenient for the forthcoming applications.
\begin{corollary}[GRH]\label{coro:coromaintheoremthatimpliesklpt}
There exists a constant $c > 0$ such that for any positive integer $n$ with $\log(n) \geq c \log\max( a, b, v)$, we have
\begin{equation*}
\sum_{i=1}^t\#\mathscr S_n(f_i) \geq \frac{n}{abv}\frac{1 + \chi\left(\frac{n - ua}{b}\right)}{(\log n)^c}.
\end{equation*}
\end{corollary}

\begin{remark}
If $\#\mathscr S_n(f_i) \neq 0$ and $\#\mathscr S_n(f_j) \neq 0$, then $f_i$ and $f_j$ must be in the same genus. Therefore, the sums in Theorem~\ref{thm:estimationofsumSi} and Corollary~\ref{coro:coromaintheoremthatimpliesklpt} are actually sums over class representatives of a single genus.
\end{remark}


\subsection{Preliminary results}
We first present a theorem that will be a central tool in the proof of Theorem~\ref{thm:estimationofsumSi}.
It is essentially~\cite[Theorem~2.1]{ABL20} with minor tweaks.
\begin{theorem}[GRH]\label{thm:ABL}
There exists a positive constant $\delta$ with the following property. Let $x \geq 2$, $b,c,d\in\Z_{>0}$, $c_0,d_0 \in \Z$, $\GCD(d_0,d) = \GCD(c_0,c) = 1$, $a_1,a_2 \in \Z\setminus\{0\}$, $\GCD(b,da_1a_2) = 1$ such that
$$Q\leq x^{1/2+\delta},\ \ \ a_1 \leq x^{1+\delta},\ \ \ a_2 \leq x^\delta,\ \ \ b,c,d \leq x^\delta.$$
Then we have
$$\sum_{\substack{q\leq Q\\ \GCD(q,a_1a_2d) = 1\\q\equiv c_0 \bmod c}}\left(
\sum_{\substack{n\leq x\\n \equiv a_1/a_2 \bmod bq\\n \equiv d_0 \bmod d}}\Lambda(n)
- \frac{1}{\phi(qbd)}\sum_{\substack{n\leq x\\\GCD(n,qbd) = 1}}\Lambda(n)
\right) \ll x^{1-\delta},$$
where $\Lambda$ is the von Mangoldt function.
\end{theorem}
\begin{proof}
Observe that if $d\mid c^\infty$, and $b = 1$, this is~\cite[Theorem~2.1]{ABL20}.
First, the condition $d\mid c^\infty$ is removed, and replaced with $\GCD(q,d) = 1$, with the following simple trick. 
Let $\delta_0$ be as in~\cite[Theorem~2.1]{ABL20}, and $\delta = \delta_0/2$.
Let $D$ be the product of prime factors of $d$ that do not divide $c$.
Since $d\mid (Dc)^\infty$ and $Dc \leq x^{\delta_0}$, we can apply~\cite[Theorem~2.1]{ABL20} and obtain that our sum (still with $b=1$) is
\begin{align*}
\sum_{\substack{c_0' \bmod Dc\\c_0' \equiv c_0 \bmod c\\\GCD(c_0',D)=1}}
\sum_{\substack{q\leq Q\\(q,a_1a_2) = 1\\q\equiv c_0' \bmod Dc}}\left(
\sum_{\substack{n\leq x\\n \equiv a_1/a_2 \bmod q\\n \equiv d_0 \bmod d}}\Lambda(n)
- \frac{1}{\phi(qd)}\sum_{\substack{n\leq x\\\GCD(n,qd) = 1}}\Lambda(n)
\right) \ll \varphi(D)x^{1-\delta_0} \leq x^{1-\delta}.
\end{align*}
This proves the theorem in the case $b=1$.\\

Second, let us deal with the case where $b\neq 1$. Let $\delta_1$ be such that the theorem holds with $b=1$. Let $\delta_2 = \delta_1/5$, and assume the conditions of the theorem are met for $\delta_2$. In particular, $bdx^{3\delta_2} \leq x^{\delta_1}$, and for any $\varepsilon>0$ we have
\begin{align*}
&\sum_{\substack{q_b \leq Q\\ q_b \mid b^\infty\\\GCD(q_b,c)=1}}
\sum_{\substack{q\leq Q/q_b\\\GCD(q,a_1a_2db) = 1\\q\equiv c_0/q_b \bmod c}}\left(
\sum_{\substack{n\leq x\\n \equiv a_1/a_2 \bmod q\\
n \equiv a_1/a_2 \bmod bq_b\\n \equiv d_0 \bmod d}}\Lambda(n)
- \frac{1}{\phi(qq_bbd)}\sum_{\substack{n\leq x\\\GCD(n,qbd) = 1}}\Lambda(n)
\right) \\
&\ll \sum_{\substack{q_b \leq x^{3\delta_2}\\ q_b \mid b^\infty\\\GCD(q_b,c)=1}}x^{1-\delta_1}
+ \sum_{\substack{x^{3\delta_2} \leq q_b \leq Q\\ q_b \mid b^\infty\\\GCD(q_b,c)=1}}\sum_{\substack{q\leq Q/q_b}}\left(
\sum_{\substack{n\leq x\\
n \equiv a_1/a_2 \bmod qbq_b}}\Lambda(n)
\right) \\
&\ll_\varepsilon x^{1-\delta_1+3\delta_2}
+ \sum_{\substack{x^{3\delta_2} \leq q_b \leq Q\\ q_b \mid b^\infty\\\GCD(q_b,c)=1}}\sum_{\substack{q\leq Q/q_b}}
\left(\frac{x}{\phi(qbq_b)} + x^{1/2 + \varepsilon}\right)\\
&\leq x^{1-2\delta_2}
+ x^{1-3\delta_2}\sum_{\substack{x^{3\delta_2} \leq q_b \leq Q\\ q_b \mid b^\infty\\\GCD(q_b,c)=1}}\sum_{\substack{q\leq Q/q_b}}
\frac{1}{\phi(q)}
+x^{1/2 + \varepsilon}
\sum_{\substack{x^{3\delta_2} \leq q_b \leq Q\\ q_b \mid b^\infty\\\GCD(q_b,c)=1}}Q/q_b
\\
&\leq x^{1-2\delta_2}
+ x^{1-3\delta_2}t(Q;b,c)\log(Q) + x^{1+\varepsilon - 2\delta_2}t(Q;b,c),
\end{align*}
where $t(Q;b,c)$ is the number of positive integers at most $Q$ whose prime factors divide $b$ but not $c$.
Using the estimates $t(Q;b,c) \leq (\log Q)^{\omega(b)}$ and $\omega(b)\ll\frac{\log(b)}{\log \log(b)}$, get that $t(Q;b,c)\ll_\varepsilon x^{\delta_2 + \varepsilon}$, and we deduce that there is a constant $\delta_3>0$ such that the above is dominated by $x^{1-\delta_3}$.
\end{proof}

We now prove an elementary lemma, which is a slight generalisation of~\cite[Lemma~3]{Hooley57} or~\cite[Lemma~5.2]{ABL20}.

\begin{definition}
For any positive integers $m$ and $s$, and any $x > 0$, let $$\Phi_{m,s}(x) = \sum_{\substack{d \leq x\\\GCD(d,m)  = 1}} \frac{\chi(d)}{\phi(sd)}.$$
\end{definition}



\begin{lemma}\label{lem:estimatingphimsx}
For any $\varepsilon>0$,
\begin{align*}
\Phi_{m,s}(x) = &\frac{1}{\phi(s)}L(1,\chi)C(\chi,m,s) 
+ O_\varepsilon\left(\frac{f_\chi^{1/2+\varepsilon}(ms x)^{\varepsilon}(\log x)^{\omega(s;m)}}{x}\right),
\end{align*}
where 
$C(\chi,m,s)$ is as in Theorem~\ref{thm:estimationofsumSi}, where $\omega(s;m)$ is the number of prime divisors of $s$ that do not divide $m$.
\end{lemma}


\begin{proof} Decomposing $d$ as $d_s d$ where $(d,s) = 1$, then applying~\cite[Lemma~5.2]{ABL20}, we get
\begin{align*}
\Phi_{m,s}(x) &=
\sum_{\substack{d_s \leq x\\d_s\mid s^\infty,\,\GCD(d_s,m)  = 1}}\frac{\chi(d_s)}{\phi(sd_s)}\sum_{\substack{d \leq x/d_s\\\GCD(d,ms)  = 1}} \frac{\chi(d)}{\phi(d)}\\
& = \sum_{\substack{d_s \leq x\\d_s\mid s^\infty,\,\GCD(d_s,m)  = 1}}\frac{\chi(d_s)}{\phi(sd_s)}\left(L(1,\chi)C(\chi,ms,1) + O\left(\frac{\tau(ms)f_\chi^{1/2}\log (f_\chi )\log(x/d_s)}{x/d_s}\right)\right),
\end{align*}
where $\tau(n)$ is the number of divisors of $n$.
The error term is dominated by
\begin{align*}
\frac{\tau(ms)f_\chi^{1/2}\log f_\chi \log(x)}{x}
\sum_{\substack{d_s \leq x\\d_s\mid s^\infty,\,\GCD(d_s,m)  = 1}}\frac{d_s}{\phi(sd_s)} =\frac{\tau(ms)f_\chi^{1/2}\log (f_\chi) t(x;s,m)\log(x)}{\phi(s)x},
\end{align*}
where $t(x;s,m)$ is the number of integers at most $x$ whose prime factors divide $s$ but not $m$. In particular, $t(x;s,m) \leq \log(x)^{\omega(s;m)}$.
The main term in the lemma follows from the equality
\begin{align*}
\sum_{\substack{d_s \mid s^\infty\\\GCD(d_s,m)  = 1}}\frac{\chi(d_s)}{\phi(sd_s)}
= \frac {1} {\phi(s)}\sum_{\substack{d_s \mid s^\infty\\\GCD(d_s,m)  = 1}}\frac{\chi(d_s)}{d_s}
= \frac {1} {\phi(s)}\prod_{p\mid s,p\nmid m}\left(1- \frac{\chi(p)}{p}\right)^{-1}.
\end{align*}
The latter sum can be cut off to $d_s < x$ with an error dominated by
\begin{align*}
\frac{1}{\phi(s)}\sum_{\substack{d_s \mid s^\infty\\\GCD(d_s,m)  = 1\\d_s > x}}\frac{1}{d_s}.
\end{align*}
An elementary recursion on the number of prime factors of $s$ not dividing $m$ yields
\begin{align*}
\sum_{\substack{d_s \mid s^\infty\\\GCD(d_s,m)  = 1\\d_s > x}}\frac{1}{d_s} \leq \frac{1}{x}(\log x)^{\omega(s)}.
\end{align*}
%
%
Overall, the contributed error is dominated by
\begin{align*}
\frac{1}{\phi(s)}L(1,\chi)C(\chi,ms,1)\sum_{\substack{d_s \mid s^\infty\\(d_s,m)  = 1\\d_s > x}}\frac{1}{d_s}
&\ll \frac{\log (f_\chi)  \log \log(ms)(\log x)^{\omega(s)}}{\phi(s)x},
\end{align*}
where we used the estimates $L(1,\chi) = O(\log f_\chi)$ and $C(\chi,m,s) = O(\log \log m)$, a consequence of the formula $\prod_{\ell \mid n}\left(1-\frac{1}{\ell}\right)^{-1} = O(\log \log n)$.
\end{proof}

\begin{corollary}\label{cor:estimatingphimsx}
For any $\varepsilon>0$, there exists $\delta$ such that if $f_\chi,s\leq x^{\delta}$, then
\begin{align*}
\Phi_{m,s}(x) = &\frac{1}{\phi(s)}L(1,\chi)C(\chi,m,s) 
+ O_\varepsilon\left(\frac{(mx)^\varepsilon}{x}\right).
\end{align*}
\end{corollary}
\begin{proof}
It follows from Lemma~\ref{lem:estimatingphimsx} and the estimate $\omega(s) = O\left( \frac{\log s}{\log \log s}\right)$.
\end{proof}

\subsection{Proof of Theorem~\ref{thm:estimationofsumSi}}
Following a classical approach to the Hardy and Littlewood problem, and using that $v = f_\chi f_0$ and $\GCD(n-au,f_0)=1$, we have
\begin{align*}
\sum_{i=1}^t\#\mathscr S_n(f_i) &= \sum_{\substack{\ell \leq n/a\\ \ell\text{ prime}\\ \ell \equiv n/a \bmod b\\ \ell \equiv u \bmod v}}\sum_{i=1}^t r\left(\frac{n-a\ell}{b};f_i\right)
= w \sum_{\substack{\ell \leq n/a\\ \ell\text{ prime}\\ \ell \equiv n/a \bmod b\\ \ell \equiv u \bmod v}}
\sum_{d \mid \frac{n-a\ell}{b}}\chi(d).
\end{align*}
The rest of the proof is dedicated to estimating the related sum
$$\sum_{\substack{\ell \leq n/a\\ \ell \equiv n/a \bmod b\\ \ell \equiv u \bmod v}}\Lambda(\ell)
\sum_{d \mid \frac{n-a\ell}{b}}\chi(d),$$
where again, $\Lambda$ is the von Mangoldt function.
The theorem then follows by partial summation.\\

Let $\delta > 0$ be a parameter to be adjusted.
For any $\varepsilon >0$,
The terms where $\ell < n^{1-\delta}$ can trivially be  bounded by $O_\varepsilon(x^{1-\delta+\varepsilon})$. 
We can deal similarly with the terms where $\ell > n/a -  n^{1-\delta}$
Therefore, it is sufficient to consider sums of the form
\begin{align}\label{eq:mainthiswewanttoestimateXY}
&\sum_{\substack{Y < \ell \leq X \\ \ell \equiv u \bmod v \\ \ell \equiv n/a \bmod b}}\Lambda(\ell)
\sum_{d \mid \frac{n-a\ell}{b}}\chi(d),
\end{align}
where $n^{1-\delta} \leq Y < X \leq n/a - n^{1-\delta}$, and $n/(2a) < X$.
Since $(a,n)=1$, the terms where $\GCD(d,an) \neq 1$ can be discarded, and the remaining 
terms are distributed into three parts as
\begin{align*}
\sum_{\substack{Y < \ell \leq X \\ \ell \equiv u \bmod v \\ \ell \equiv n/a \bmod b}}\Lambda(\ell)
\sum_{\substack{d \mid \frac{n-a\ell}{b}\\\GCD(d,an) =1}}\chi(d)
&= \sum_{\substack{Y < \ell \leq X \\ \ell \equiv u \bmod v \\ \ell \equiv n/a \bmod b}}\Lambda(\ell)\left(
\sum_{\substack{d \leq D\\ db \mid n-a\ell\\\GCD(d,an) =1}}\chi\left(\frac{n-a\ell}{db}\right)+
\sum_{\substack{d < \frac{n-a\ell}{Db}\\ db \mid n-a\ell\\\GCD(d,an) =1}}\chi(d)
\right)\\
&= S_1 + S_2 + S_3,
\end{align*}
where $D = X^{1/2}$, and
\begin{align*}
S_1 &= \sum_{\substack{d \leq D\\\GCD(d,an) = 1}}\sum_{\substack{Y < \ell \leq X \\ \ell \equiv u \bmod v\\   \ell \equiv n/a \bmod db}}\Lambda(\ell)
\chi\left(\frac{n-a\ell}{db}\right),\\
%
S_2 &= \sum_{\substack{d \leq \frac{n-Xa}{Db}\\\GCD(d,an) = 1}}\chi(d)\sum_{\substack{Y < \ell \leq X \\ \ell \equiv u \bmod v\\   \ell \equiv n/a \bmod db}}\Lambda(\ell),\\
S_3 &= \sum_{\substack{\frac{n-Xa}{Db} < d \leq \frac{n-Ya}{Db}\\\GCD(d,an) = 1}}\chi(d)\sum_{\substack{Y < \ell \leq \frac{n-bdD}{a} \\ \ell \equiv u \bmod v\\   \ell \equiv n/a \bmod db}}\Lambda(\ell).
\end{align*}
\\

\noindent\emph{Estimation of $S_1$.}
Decomposing $d$ as $dd_\chi$ where $(d,f_\chi) = 1$ and $d_\chi \mid f_\chi^\infty$, we have
\begin{align*}
S_1
&=\sum_{\substack{d_\chi \leq D \\ d_\chi \mid f_\chi^\infty\\\GCD(d_\chi,an)=1}}\sum_{\substack{d \leq D/d_\chi\\\GCD(d,f_\chi an) = 1}}\overline\chi(d)\sum_{\substack{Y < \ell \leq X \\ \ell \equiv u \bmod v\\   \ell \equiv n/a \bmod dd_\chi b}}\Lambda(\ell)
\chi\left(\frac{n-a\ell}{d_\chi b}\right)\\
&= \sum_{\substack{d_\chi \leq D \\ d_\chi \mid f_\chi^\infty\\\GCD(d_\chi,an)=1}}
\sum_{y \bmod f_\chi}\overline\chi(y)
\sum_{x \bmod f_\chi}\chi(x)
\sum_{\substack{d \leq D/d_\chi\\\GCD(d,an) = 1 \\ d \equiv y \bmod f_\chi}}
\sum_{\substack{Y < \ell \leq X \\ \ell \equiv u \bmod v\\ \ell \equiv n/a \bmod db \\ \ell \equiv (n-xbd_\chi)/a \bmod f_\chi d_\chi}}\Lambda(\ell).
\end{align*}
Since $\GCD(n-au,f_0) = 1$, the terms where $\GCD(d,f_0) \neq 1$ are zero.
Applying Theorem~\ref{thm:ABL} and the prime number theorem, we have
\begin{align*}
S_1 &= \sum_{\substack{d_\chi \leq D \\ d_\chi \mid f_\chi^\infty\\\GCD(d_\chi,an)=1}}
\sum_{y \bmod  f_\chi}\overline\chi(y)
\sum_{\substack{x \bmod f_\chi\\ua\equiv n - xbd_\chi \bmod \GCD(f_\chi d_\chi,v)}}\chi(x)
\sum_{\substack{d \leq D/d_\chi\\\GCD(d, anf_0) = 1 \\ d \equiv y \bmod  f_\chi}}
\sum_{\substack{Y < \ell \leq X \\ \ell \equiv u \bmod v\\ \ell \equiv n/a \bmod db \\ \ell \equiv (n-xbd_\chi)/a \bmod f_\chi d_\chi}}\Lambda(\ell)\\
&\approx\sum_{\substack{d_\chi \leq D \\ d_\chi \mid f_\chi^\infty\\\GCD(d_\chi,an)=1}}
\sum_{y \bmod  f_\chi}\overline\chi(y)
\sum_{\substack{x \bmod f_\chi \\ua\equiv n - xbd_\chi \bmod f_\chi\GCD(d_\chi,f_0)}}\chi(x)
\sum_{\substack{d \leq D/d_\chi\\\GCD(d, anf_0) = 1 \\ d \equiv y \bmod  f_\chi}}
\frac{\sum_{\substack{Y < \ell \leq X \\ (\ell, dbvf_\chi) = 1}}\Lambda(\ell)}{\phi(db \mathrm{lcm}(v,f_\chi d_\chi))}\\
&\approx\sum_{\substack{d_\chi \leq D \\ d_\chi \mid f_\chi^\infty\\\GCD(d_\chi,an)=1}}
\sum_{y \bmod f_\chi}\overline\chi(y)
\sum_{\substack{x \bmod f_\chi \\ua\equiv n - xbd_\chi \bmod f_\chi\GCD(d_\chi,f_0)}}\chi(x)
\sum_{\substack{d \leq D/d_\chi\\\GCD(d, anf_0) = 1 \\ d \equiv y \bmod  f_\chi}}
\frac{X-Y}{\phi(db f_\chi\mathrm{lcm}(f_0,d_\chi))},
\end{align*}
where the error introduced by the approximations is dominated by
\begin{align*}
&\sum_{\substack{d_\chi \leq D \\ d_\chi \mid f_\chi^\infty\\\GCD(d_\chi,an)=1}}
\sum_{\substack{y \bmod  f_\chi\\\GCD(y,f_\chi) = 1}}
\sum_{\substack{x \bmod f_\chi\\\GCD(x,f_\chi) = 1\\ua\equiv n - xbd_\chi \bmod f_\chi\GCD(d_\chi,f_0)}}
X^{1-\delta_0}
\leq X^{1-\delta_0}\phi(f_\chi)^2t(D;f_\chi,an),
\end{align*}
where $\delta_0$ is the constant of Theorem~\ref{thm:ABL}, and $t(D;f_\chi,an)$ is the number of integers at most $D$ whose prime factors divide $f_\chi$ but not $an$.
Reorganising the terms, our estimation of $S_1$ becomes
$$S_1 \approx (X-Y)\sum_{\substack{d_\chi \leq D \\ d_\chi \mid f_\chi^\infty\\\GCD(d_\chi,an)=1}}
\left(\sum_{\substack{x \bmod f_\chi \\ua\equiv n - xbd_\chi \bmod f_\chi\GCD(d_\chi,f_0)}}\chi(x)\right)
\left(\sum_{\substack{d \leq D/d_\chi\\\GCD(d, anf_0) = 1}}
\frac{\overline\chi(d)}{\phi(dbf_\chi \mathrm{lcm}(f_0,d_\chi))}\right).$$
Let us focus on the first inner sum. Let $m = \gcd(n-ua, f_\chi\GCD(d_\chi,f_0),d_\chi)$. We have
\begin{align*}&
\sum_{\substack{x \bmod f_\chi \\xbd_\chi\equiv n-ua \bmod f_\chi\GCD(d_\chi,f_0)}}\chi(x)\\
&=
\begin{cases}
\sum_{\substack{x \bmod f_\chi \\x\equiv (n - ua)/bd_\chi \bmod f_\chi\GCD(d_\chi,f_0)/m}}\chi(x)&\text{ if } \substack{\GCD(d_\chi,f_\chi\GCD(d_\chi,f_0)) = m \\\text{and }\GCD(n-ua,f_\chi\GCD(d_\chi,f_0)) = m,}\\
0&\text{ otherwise.}\\
\end{cases}
\end{align*}
In the situation where $\GCD(d_\chi,f_\chi\GCD(d_\chi,f_0))=\GCD(n-ua,f_\chi\GCD(d_\chi,f_0))$, let $\alpha = (n - ua)/bd_\chi \bmod f_\chi\GCD(d_\chi,f_0)/m$. From~\cite[(3.9)]{IK04}, we have
\begin{align*}
\sum_{\substack{x \bmod f_\chi \\x\equiv \alpha \bmod f_\chi\GCD(d_\chi,f_0)/m}}\chi(x)
&= \sum_{\substack{y \bmod m/\GCD(d_\chi,f_0)}}\chi(\alpha + y(f_\chi\GCD(d_\chi,f_0)/m))\\
&= \frac{m}{f_\chi\GCD(d_\chi,f_0)}\sum_{\substack{y \bmod f_\chi}}\chi(\alpha + y(f_\chi\GCD(d_\chi,f_0)/m))\\
&= \begin{cases}
\chi(\alpha) &\text{ if } m = \GCD(d_\chi,f_0),\\
0&\text{otherwise.}
\end{cases}
\end{align*}
In summary, and using that $\GCD(n-ua,f_0) = 1$, we get
\begin{align*}
\sum_{\substack{x \bmod f_\chi \\xbd_\chi\equiv n-ua \bmod v}}\chi(x)
&=
\begin{cases}
\chi\left(\frac{n - ua}{b}\right)&\text{ if }\GCD(n-ua,f_\chi) = 1\text{ and }d_\chi = 1,\\
0&\text{ otherwise.}\\
\end{cases}
\end{align*}
We deduce that our estimation of $S_1$ is 
\begin{align*}
S_1 &\approx (X-Y)\frac{1}{\phi(f_\chi f_0)}
\chi\left(\frac{n - ua}{b}\right)
\left(\sum_{\substack{d \leq D\\\GCD(d, anf_0) = 1}}
\frac{\chi(d)}{\phi(db)}\right).
\end{align*}
\\

\noindent\emph{Estimation of $S_2$.} 
Similarly, using $\frac{n-Xa}{Db} \leq \frac{n}{Db} \leq \frac{2aX^{1/2}}{b}$, we can apply Theorem~\ref{thm:ABL} and get up to an admissible error
\begin{align*}
S_2 &= \sum_{y \bmod f_\chi}\chi(y)\sum_{\substack{d \leq \frac{n-Xa}{Db}\\\GCD(d,anf_0) = 1\\d \equiv y \bmod f_\chi}}\sum_{\substack{Y < \ell \leq X \\ \ell \equiv u \bmod f_\chi f_0\\   \ell \equiv n/a \bmod db}}\Lambda(\ell)\\
&\approx \sum_{y \bmod f_\chi}\chi(y)\sum_{\substack{d \leq \frac{n-Xa}{Db}\\\GCD(d,anf_0) = 1\\d \equiv y \bmod f_\chi}}\frac{X-Y}{\phi(dbf_\chi f_0)}\\
&\approx (X-Y)\frac{1}{\phi(f_\chi f_0)}\sum_{\substack{d \leq \frac{n-Xa}{Db}\\\GCD(d,anf_0) = 1}}\frac{\chi(d)}{\phi(db)}.
\end{align*}

\noindent\emph{Main term of the estimation.} Anticipating that $S_3$ will disappear in the error term, we get that the main term of our estimation of the sum~\eqref{eq:mainthiswewanttoestimateXY} is
$$S_1+S_2 \approx (X-Y)\frac{1}{\phi(b)}\frac{1}{\phi(f_\chi f_0)}\left(1 + \chi\left(\frac{n - ua}{b}\right)\right)L(1,\chi)C(\chi,anf_0,b),$$
From Corollary~\ref{cor:estimatingphimsx}, for any $\varepsilon > 0$, the error introduced in the above estimation is dominated by
$$\frac{X-Y}{\phi(f_\chi f_0)}\frac{(anf_0D')^\varepsilon}{D'},$$
for $D' = \min \left(D,\frac{n-Xa}{Db}\right) \geq n^{\frac 1 2 - \frac{3\delta}{2}}$.
\\
 

\noindent\emph{Estimation of $S_3$.}
We now prove that the third and last term is absorbed in the error. We have
\begin{align*}
S_3 &= \sum_{\substack{\frac{n-Xa}{Db} < d \leq \frac{n-Ya}{Db}\\\GCD(d,an) = 1}}\chi(d)\sum_{\substack{Y < \ell \leq \frac{n-bdD}{a} \\ \ell \equiv u \bmod v\\   \ell \equiv n/a \bmod d b}}\Lambda(\ell)
= 
\sum_{y \bmod f_\chi}\chi(y)
\sum_{\substack{\frac{n-Xa}{Db} < d \leq \frac{n-Ya}{Db}\\\GCD(d, anf_0) = 1 \\ d \equiv y \bmod f_\chi}}
\sum_{\substack{Y < \ell \leq \frac{n-bdD}{a} \\ \ell \equiv u \bmod v\\ \ell \equiv n/a \bmod db}}\Lambda(\ell).
\end{align*}
To split this sum into manageable components, let \begin{align*}
T_\chi(s,t,y,x) &= 
\sum_{\substack{y \bmod f_\chi}}\chi(y)
\sum_{\substack{s < d \leq t\\\GCD(d, anf_0) = 1 \\ d \equiv y \bmod f_\chi}}
\sum_{\substack{y < \ell \leq x \\ \ell \equiv u \bmod v\\ \ell \equiv n/a \bmod db}}\Lambda(\ell).
\end{align*}
Assuming that $\frac{n-Xa}{Db}\leq  s \leq t \leq n/Db$ and $x \geq Y$, and recalling that $Y \geq n^{1-\delta}$ where $\delta$ is adjustable, we can ensure that Theorem~\ref{thm:ABL} and Corollary~\ref{cor:estimatingphimsx} apply in the following estimation: 
\begin{align*}
T_\chi(s,t,y,x) &= 
\sum_{y \bmod f_\chi}\chi(y)
\sum_{\substack{s < d \leq t\\\GCD(d, anf_0) = 1 \\ d \equiv y \bmod f_\chi }}
\frac{x-y}{\phi(vdb)} + O(\phi(f_{\chi})x^{1-\delta_0})\\
&= (x-y)
\sum_{\substack{s < d \leq t\\\GCD(d, anf_0) = 1 }}
\frac{\chi(d)}{\phi(vdb)} + O(\phi(f_{\chi})x^{1-\delta_0})\\
&= (x-y)
(\Phi_{anf_0,vb}(t)-\Phi_{anf_0,vb}(s)) + O(\phi(f_{\chi})x^{1-\delta_0})\\
&\ll (x-y)\frac{(anf_0s)^\varepsilon}{s} +\phi(f_{\chi})x^{1-\delta_0}.
\end{align*}

Also, if $\mathbf 1_\chi$ is the principal character of conductor $f_\chi$,
\begin{align*}
T_{\mathbf 1_\chi}(s,t,y,x) &= 
\sum_{\substack{s < d \leq t\\(d, anvb) = 1}}
\sum_{\substack{y < \ell \leq x \\ \ell \equiv u \bmod v\\ \ell \equiv n/a \bmod db}}\Lambda(\ell)\\
&= 
\sum_{\substack{s < d \leq t\\\GCD(d, anvb) = 1}}\frac{x-y}{\phi(vdb)}
+ O(x^{1-\delta_0})\\
&\ll  (x-y)\log(t/s) + x^{1-\delta_0}.
\end{align*}
Let $\Delta = n^{-\delta_0/2}$.
We now split the $\ell$-sum in $S_3$ into $k$ intervals of the form $(L_i, L_{i+1}]$ with $L_1 = Y$, $L_{i+1} = (1-\Delta)L_i +\Delta n/a$, and cropping the last interval so that $\left(Y,\frac{n-bdD}{a} \right] = \bigcup_{i=1}^k (L_i, L_{i+1}]$. 
Let $F(\ell) = \frac {n}{bD} - \ell \frac {a}{bD}$.
We have $F(L_{i+1})/F(L_i) = 1-\Delta$. In particular, $k$ is the smallest integer such that $(1-\Delta)^{k-1} \leq F(X)/F(Y)$. We can assume that $X \leq n/a - 1$, and deduce that $k = O(\Delta^{-1} \log(n))$.
Now, we have
$$S_3 \ll \sum_i \left(T_\chi\left(\frac{n-Xa}{bD},F(L_{i+1}),L_i,L_{i+1}\right) + T_{\mathbf 1_\chi}(F(L_{i+1}),F(L_i),L_i,L_{i+1})\right).$$
On one hand,
\begin{align*}
\sum_{i=1}^{k} T_{\mathbf 1_\chi}(F(L_{i}),F(L_{i-1}),L_{i-1},L_{i}) &\ll \sum_{i=1}^k \left((L_{i}-L_{i-1})\log(F(L_{i-1})/F(L_{i})) + L_{i}^{1-\delta_0}\right)\\
&= -\log(1-\Delta)\sum_{i=1}^k (L_{i}-L_{i-1}) + kX^{1-\delta_0}\\
&\ll \Delta(X-Y) + \Delta^{-1}\log(n)X^{1-\delta_0}\\
&\leq X^{1-\delta_0/2}(1 + \log(n))\\
&\ll n^{1-\delta}.
\end{align*}
On the other hand, writing $s = (n-Xa)/bD \geq \frac{a}{bD}n^{1-\delta} \geq \frac{n^{1/2-\delta}}{b}$,
\begin{align*}
\sum_{i=1}^k T_\chi(s,F(L_{i}),L_{i-1},L_{i}) &\ll 
\sum_{i=1}^k
\left((L_{i}-L_{i-1})\frac{(anf_0s)^\varepsilon}{ s} +\phi(f_{\chi})L_i^{1-\delta_0}\right)\\
&= (X-Y)\frac{(anf_0s)^\varepsilon}{ s}  + X^{1-\delta_0/2}\log(n)\phi(f_{\chi})\\
&\ll n^{1/2 + \delta+2\varepsilon}b(af_0)^\varepsilon  + n^{1-\delta_0/2}\log(n)\phi(f_{\chi}).
\end{align*}
This proves that choosing $\varepsilon$ and $\delta$ appropriately, $S_3$ is absorbed in the error term. It concludes the proof of Theorem~\ref{thm:estimationofsumSi}.
\qed

\section{Solving equations of the form $\det(\gamma)^2f(s,t) + bf^{\gamma}(x,y) = n$}\label{sec:solving}

\noindent 
Let $b$ and $n$ be positive integers, and $f$ a primitive, positive definite, integral, binary quadratic form whose discriminant is fundamental. Let $\gamma \in M_{2\times 2}(\Z)$ be a matrix of rank $2$.
In this section, we focus on the problem of finding integer solutions of the equation
$$\det(\gamma)^2f(s,t) + bf^{\gamma}(x,y) = n.$$
More precisely, we prove the following theorem.

\begin{theorem}[GRH]\label{thm:solvingthecomplicatedequation}
There exists a constant $c>0$ and an algorithm $\mathscr A$ such that the following holds.
Let $b$ and $n$ be positive integers, and $f$ a reduced, primitive, positive definite, integral, binary quadratic form whose discriminant is fundamental. Let $\gamma \in M_{2\times 2}(\Z)$ of rank~$2$ and content $1$.
Suppose that 
the factorisation of $\det(\gamma)$ is known, that $\det(\gamma)$, $\disc(f)$ and $b$ are pairwise coprime, and that $\GCD(\det(\gamma)b,n) = 1$.
Suppose that $\log n \geq \max(c \cdot \log b, \log (\det(\gamma))^c, \disc(f)^c)$, and either
\begin{enumerate}
\item \label{item:conditionnumberdivisorsn} 
$\log n \geq \omega(n)^{c}$, or 
\item \label{item:conditionprimedivisorsn} 
the prime divisors of $n$ are larger than $\disc(f)^c$, $(\log\log b)^c$ and $\log(\det(\gamma))^c$. 
\end{enumerate}
Then $\mathscr A(f,\gamma,b,n)$ returns an integer solution $(s,t,x,y)\in \Z^4$ of the equation
$$\det(\gamma)^2f(s,t) + bf^{\gamma}(x,y) = n,$$
provided that the equation has a solution modulo $\disc(f^
\gamma)$ for which $f(s,t)$ is invertible modulo $\disc(f)$.
The algorithm runs in expected polynomial time in $\disc(f)$, $\length(\gamma)$, $\log n$, 
and the output is random with min-entropy $\Omega(\log n)$.
\end{theorem}


Cornacchia's algorithm allows to solve equations of the form
$$f(s,t) = z,$$
in time polynomial in $\disc(f)$ and $\log z$ when the factorisation of $z$ is known and a solution exists.
We are therefore led to study the solutions $(z,x,y)$ of the equation
\begin{equation}\label{eq:detgammazbfgammaxy}
\det(\gamma)^2 z + b f^\gamma(x,y) = n,
\end{equation}
where the factorisation of $z$ is known and $z>0$.
Factoring is hard, but primality testing is easy, so we will simply look for solutions where $z$ is prime.
Having $z$ prime has another advantage: if $\chi$ is the Kronecker symbol of modulus $\disc(f)$, 
the condition $\chi(z) = 1$ ensures that there is a solution $f'(s,t) = z$ for some $f'$ of same discriminant as $f$. 
Replacing this condition by $z \equiv u \bmod \disc(f)$ for any $u$ represented by $f$ ensures that $z$ is represented by some $f'$ in the same genus as $f$. Ensuring that $z$ is represented by $f$ itself will require additional tricks.

\subsection{Solutions of $a z + b g(x,y) = n$}
First, let us lift the delicate primality condition on $z$.
The following proposition allows to sample random solutions of Equation~\eqref{eq:detgammazbfgammaxy} if $z$ is only required to be a positive integer.

\begin{proposition}\label{prop:samplingnbfxya}
Let $g$ be a primitive, positive definite, integral, binary quadratic form.
Let $a,b,n$ be positive integers, and suppose that $a$ divides $\disc(g)$ and $\GCD(a,2bn) = 1$.
Let $X$ be the set of integral solutions $(z,x,y)$ of the equation
$a z + b g(x,y) = n$, with $z > 0$.
If there exists a solution modulo $a$, then $X$ is a disjoint union of $2^{\omega(a)}$ sets $X_i$, with
$$\#X_i = \frac{ \pi n}{ba\Vol(g)} + O\left(\left(\frac n b\right)^{1/2} + a\Vol(g)\right),$$
and knowing the factorisation of $a$ allows to 
sample uniformly from any $X_i$ in polynomial time.
\end{proposition}
\begin{proof}
First consider the equation modulo $a$. Since $a$ divides the discriminant of $g(x,y)$, the latter polynomial splits modulo $a$ and the equation becomes
$$\varepsilon\cdot(\alpha x + \beta y)^2 = n/b \bmod a,$$
where $\varepsilon$, and a least one of $\alpha$ and $\beta$, are invertible. 
Now, the element $n/{b\varepsilon} \bmod a$ is invertible, so it either has no square root (in which case $X$ is empty), or it has $2^{\omega(a)}$ distinct square roots.
Suppose it has square roots. There is a sublattices $\Lambda$ of index $a$ in $\Z^2$ such that the space of admissible pairs $(x,y)$ is the disjoint union $\bigsqcup_{i = 1}^\delta(\Lambda + v_i)$, where the $v_i$-vectors are representative solutions for the $2^{\omega(a)}$ roots modulo $a$.
Accounting for the condition $z>0$, it remains to count for each translated lattice $\Lambda + v_i$ the number of points $(x,y) \in \Lambda +  v_i$ such that
$$g(x,y) < n/b.$$
From Lemma~\ref{lem:balldim2}, it is equal to
$$\frac{\pi n}{b\Vol(\Lambda)\Vol(g)} + O\left(\left(\frac n b\right)^{1/2} + \Vol(\Lambda)\Vol(g)\right) = \frac{\pi n}{ba\Vol(g)} + O\left(\left(\frac n b\right)^{1/2} + a\Vol(g)\right).$$
Let $X_i$ be the solutions stemming from $\Lambda +  v_i$. 
Given the factorisation of $a$, one can compute all the square roots of $n/{b\varepsilon} \bmod a$. Therefore,
to sample uniformly in $X_i$, apply Lemma~\ref{lem:samplingbinary} to sample uniformly a point in the intersection of $\Lambda + v_i$ and the ellipsoid $g(x,y) < n/b$.
\end{proof}

%


\subsection{Randomisation in the genus} Proposition~\ref{prop:samplingnbfxya} tells us that integer solutions of Equation~\eqref{eq:detgammazbfgammaxy} can be sampled uniformly (up to a small error). We would then be done if a large proportion of these have a $z$-value which is a prime represented by $f$. Unfortunately, it is hard to control the primality and representability of these solutions when the form $f^\gamma$ is fixed. Theorem~\ref{thm:estimationofsumSi} only gives information about the family of equations where $f^\gamma$ is replaced by any form in its genus. Luckily, we can randomise within the genus thanks to the following two lemmata. Their proofs use the classical correspondence between binary quadratic forms and ideals in quadratic orders; for an account of this theory, we refer the reader to~\cite[Section~7]{Cox11}.
%

Lemma~\ref{lem:BRepresentedByAnyClass} below is useful to deal with forms $f^\gamma$ of large discriminant.

\begin{lemma}[GRH]\label{lem:BRepresentedByAnyClass}
For any discriminant $d$ and positive integer $m$,
there exists an integer $B$ coprime to $md$ such that any primitive binary quadratic form of discriminant $d$ represents a divisor of $B$, and $\log B = O_\varepsilon\left((\log |d|\cdot \left((\log |d|)^{2+\varepsilon} + \omega (m)^{1+\varepsilon}\right)\right)$. There is an algorithm which samples a form uniformly distributed in the class group, together with a representation by this class of a divisor of $B$, in time polynomial in $\log|d|$ and $\log m$.
\end{lemma}

\begin{proof}
We utilise the correspondence between classes of binary quadratic forms and ideal classes in quadratic number fields, and the fact that a form represents $n>0$ if and only if the corresponding ideal class contains an ideal of norm $n$.
The key is the fact that there is a small bound $C$ such that the ideals of prime norm at most $C$ constitute a generating set of the class group so that the Cayley graph is highly connected: any two vertices are connected by a path of length at most $D = O(\log h(d)) = O(\log |d|)$.
From~\cite[Theorem~1.1]{jmv:asiacrypt} and~\cite[Corollary~1.3]{jmv:asiacrypt}, one can choose $C = O_\varepsilon((\log |d|)^{2+\varepsilon})$ for any $\varepsilon>0$. However, to construct our $B$, we wish to consider only prime ideals coprime to $m$. The same proof as~\cite{jmv:asiacrypt} implies that we can choose $C = O_\varepsilon((\log |d|)^{2+\varepsilon} + \omega(m)^{1+\varepsilon})$: replace the estimate (in~\cite[Equation~(2.4)]{jmv:asiacrypt}, with $n=2$)
$$\sum_{\substack{N(\mathfrak p) \leq x}}\left(\chi(\mathfrak p) + \chi(\mathfrak p)^{-1}\right) = 2r\frac{x}{\log x}+ O\left(x^{1/2}\log(xd)\right)$$
with its simple corollary
$$\sum_{\substack{N(\mathfrak p) \leq x\\(\mathfrak p,m)=1}}\left(\chi(\mathfrak p) + \chi(\mathfrak p)^{-1}\right) = 2r\frac{x}{\log x} + O\left(x^{1/2}\log(xd) + \omega(m)\right).$$
We deduce that if $p_1,\dots,p_{k}$ are all the primes at most $C$ not dividing $m$, any ideal is equivalent to an ideal of norm $\prod_{i}p_i^{e_i}$ where $\sum_ie_i \leq D$. The latter is a divisor of $B = \prod_i p_i^D,$ and we have
$$\log B = D \sum_i \log p_i \leq D\pi(C)\log(C) = O_\varepsilon\left((\log |d|\cdot \left((\log |d|)^{2+\varepsilon} + \omega (m)^{1+\varepsilon}\right)\right),$$
where $\pi(C)\log(C) = O(C)$ is the prime number theorem.

To sample an ideal with norm dividing $B$ and uniformly distributed in the class group, compute a random walk of length $D$ in the (expander) Cayley graph, as in~\cite{jmv:asiacrypt}.
\end{proof}

The next lemma is similar, but mostly useful when the discriminant is small.

\begin{lemma}[GRH]\label{lem:BRepresentedByAnyClassSmallDisc}
For any discriminant $d$ and positive integer $m$,
there exists an integer $B$ coprime to $md$ such that any primitive binary quadratic form of discriminant $d$ represents a divisor of $B$, and $\log B = O_\varepsilon\left(|d|^{1/2 +\varepsilon}\log( \omega (m)+2)\right)$. There is an algorithm which samples a form uniformly distributed in the class group in time polynomial in $|d|$ and $\log m$. Given any class, one can compute a divisor of $B$ together with a representation by this class in time polynomial in $|d|$ and $\log m$.
\end{lemma}

\begin{proof}
We proceed as above, but instead of doing random walks, we use that class group computations can be done in time polynomial in the discriminant (see~\cite[Chapter~5]{Cohen13}).
As already seen in the proof of Lemma~\ref{lem:BRepresentedByAnyClass}, the class group is generated by the set $P_0$ of ideal of prime norm at most $C = O_\varepsilon((\log |d|)^{2+\varepsilon} + \omega(m)^{1+\varepsilon})$ not dividing $md$. In time polynomial in $|d|$, one can compute the class group, and find a minimal subset $P\subseteq P_0$ generating the class group. We have $\#P \leq \log(h)$, where $h = O(|d|^{1/2}\log|d|)$ is the class number.
For each class, one can compute a representative that is a product of ideals in $P$, with exponents at most the class number $h$. These representatives divide $B = \prod_{\mathfrak p \in P} N(\mathfrak p)^h$, and
$$\log B = h\sum_{\mathfrak p \in P} \log N(\mathfrak p) \leq h \log(C) \# P =O_\varepsilon(|d|^{1/2+\varepsilon}\log(\omega(m)+2)).$$
The added $2$ avoids the degeneracy at $m = 1$.
In polynomial time in $|d|$, one can sample a uniformly random ideal class. Given any class, one can return the corresponding (previously computed) representative that divides $B$.
\end{proof}

 Let $a,b,u,v$ and $n$ be positive integers, and $f$ and $g$ two primitive, positive definite, integral, binary quadratic forms.
We are looking for a solution $(z,x,y)$ to the equation
\begin{equation}\label{eq:solveinztabn}
az + bg(x, y) = n,
\end{equation}
where $z>0$ is represented by $f$.
A condition of the form $\ell \equiv u \bmod \disc(f)$ can ensure that $\ell$ is represented by the genus of $f$, but this is not enough for $\ell$ to be represented by $f$ itself. The following trick deals with this difficulty.
\begin{lemma}\label{lemma:randomizingontheleft}
There is an integer $B_0$ coprime to $2nb\disc(g)$ such that the following holds. For any $\rho \in M_{2\times 2}(\Z)$ of determinant $B_0$ and content $1$, given an integral solution $(\ell,x,y)$ of
\begin{equation*}
a\det(\rho)^2\ell + bg^{\rho}(x_0, y_0) = n,
\end{equation*}
with $\ell$ a prime represented by the genus of $f$, one can compute an integral solution $(s,t,x,y)$ of
\begin{equation*}
af(s,t) + bg(x, y) = n
\end{equation*}
in expected polynomial time in $\disc(f)$ and the size of the input.
\end{lemma}

\begin{proof}
Let $B_0$ be the integer from Lemma~\ref{lem:BRepresentedByAnyClassSmallDisc}, for $d = \disc(f)$ and $m = 2 n b \disc(g)$. 
 Let $\rho \in M_{2\times 2}(\Z)$ of determinant $B_0$ and content $1$, and suppose we have a solution $(\ell,x,y)$ of
\begin{equation*}
aB_0^2\ell + bg^{\rho}(x', y') = n,
\end{equation*}
with $\ell$ a prime represented by the genus of $f$.
Then, Cornacchia's algorithm allows to find $h$ in the genus of $f$ and integers $(s_0,t_0)$ such that $\ell = h(s_0,t_0)$. Let $k$ such that $[k]^2 = [h]^{-1}[f]$ and $(s_1,t_1)$ such that $k(s_1,t_1) = d \mid B_0$. Then, we can compute\footnote{Using the Gauss composition law, find $S,T$ such that $(k^2h)(S,T) = k(s_0,t_0)^2h(s_1, t_1)$. 
Then, find $\gamma,\gamma' \in \SL_2(\Z)$ such that $(k^2h)^\gamma$ and $g^{\gamma'}$ are reduced (see~\cite[Algorithm~5.4.2]{Cohen13}). 
Since they are in the same class and reduced, we actually have $(k^2h)^\gamma = g^{\gamma'}$.
Finally, let $(s_2,t_2) = \gamma'\gamma^{-1}(S,T)$.}
 $(s_2,t_2)$ such that
$$B_0^2 \ell = (B_0/d)^2 k(s_1,t_1)^2 h(s_0,t_0) = (B_0/d)^2 f(s_2,t_2) = f(s_2 B_0/d,t_2 B_0/d).$$
With $(s,t) = (s_2 B_0/d,t_2 B_0/d)$ and $(x,y) = \rho(x',y')$, we get $af(s,t) + bg(x, y) = n$.
\end{proof}

Therefore, it is sufficient to study solutions of Equation~\eqref{eq:solveinztabn} where $\ell$ is a prime represented by the genus of $f$, up to replacing $a$ with $aB_0^2$ and $g$ with $g^\rho$.
\\

For the rest of this section, consider all notation and conditions from Theorem~\ref{thm:solvingthecomplicatedequation}, and let $a = \det(\rho\gamma)^2$ and $g = f^{\rho\gamma}$, with $\rho$ as in Lemma~\ref{lemma:randomizingontheleft}.
By working carefully at each prime factor of $B_0$, one can craft $\rho$ in a way that ensures the local solvability of the equation $a\ell + bg(x_0, y_0) = n$.
Observe that in general, the $B_0$ constructed in Lemma~\ref{lemma:randomizingontheleft} satisfies
$$\log(B_0) = O_\varepsilon\left(|\disc(f)|^{1/2 +\varepsilon}\log (2+\log (nb\disc(g)))\right),$$
but if Condition~\eqref{item:conditionprimedivisorsn} holds, we can obtain $$\log(B_0) = O_\varepsilon\left(|\disc(f)|^{1/2 +\varepsilon}\log (2 + \log (b\disc(g)))\right)$$ by choosing $m = 2 b \disc(g)$ in the application of Lemma~\ref{lem:BRepresentedByAnyClassSmallDisc}; the condition that $(B_0,n) = 1$ is then enforced by the fact that the prime divisors of $n$ are larger than those of $B_0$.

Now, let $B$ be the integer from Lemma~\ref{lem:BRepresentedByAnyClass} for $d = \disc(g)$, and $m = n$ if Condition~\eqref{item:conditionnumberdivisorsn} holds, or $m = 1$ if Condition~\eqref{item:conditionprimedivisorsn} holds.
In either case, Lemma~\ref{lem:BRepresentedByAnyClass} ensures that $\log(B) = O((\log n)^{\frac{9+3\varepsilon}{2c}})$, either by bounding $\omega(n)$ with Condition~\eqref{item:conditionnumberdivisorsn}, or with $\omega(1) = 0$ in the other case.
Also, even in the case $m = 1$, we have $(B,n) = 1$, as Condition~\eqref{item:conditionprimedivisorsn} ensures that the prime factors of $n$ are all larger than those of $B$. Consider the sets
\begin{align*}
\mathscr X &= \{(z,x_1,y_1,h,x_0,y_0,k) \mid az + B^2bh(x_1,y_1) = n, z > 0, [k^{2}h] = [g]\text{, and $k(x_0,y_0)$ divides $B$}\},\\
X &= \{(z,x,y,[h]) \mid az + B^2bh(x,y) = n, z > 0,\text{ and $h$ is in the genus of $g$}\},\\
\mathscr S &= \{(\ell,x,y,[h]) \in X \mid \ell \text{ is prime and }\ell \equiv u \bmod \disc(f)\},
\end{align*}
where $[h]$ denotes the class of $h$ in the class group.
There is a natural surjection $\pi : \mathscr X \rightarrow X$.
\begin{lemma}\label{lem:samplinginSissufficient}
Given a tuple $T \in \mathscr X$ such that $\pi(T) \in \mathscr S$, one can compute a solution $(\ell,x,y)$ of Equation~\eqref{eq:solveinztabn} in polynomial time, where $\ell>0$ is a prime number such that $\ell \equiv u \bmod v$.
\end{lemma}
\begin{proof}
Let $T = (\ell,x_1,y_1,h,x_0,y_0,k) \in \mathscr X$ such that $\pi(T) \in \mathscr S$.
We then have
$$a\ell + B^2bh(x_1,y_1) = n.$$
Since $[k^2h] = [g]$, one can compute $x_2$ and $y_2$ such that
$k(x_0,y_0)^2h(x_1, y_1) = g(x_2, y_2)$ (see the footnote from the proof of Lemma~\ref{lemma:randomizingontheleft}). 
We obtain
$$B^2h(x_1,y_1) = (B/d)^2k(x_0,y_0)^2h(x_1, y_1) = (B/d)^2g(x_2, y_2) = g(x_2B/d, y_2B/d),$$
where $d = k(x_0,y_0)$. Then, $(\ell, x_2B/d, y_2B/d)$ is a solution of Equation~\eqref{eq:solveinztabn}.
\end{proof}

\begin{lemma}\label{lem:wecansampleinX}
Suppose $a$ is an odd prime power dividing $\disc(g)$ and coprime to $b$.
There is an algorithm that samples elements $T$ in $\mathscr X$ such that $\pi(T)$ is close to uniformly distributed in $X$, and runs in expected polynomial time. More precisely, the probability of any $x \in X$ is between $1/(2\#X)$ and $3/(2\#X)$.
\end{lemma}
\begin{proof}
Generate a uniformly random class $[k]$ in $\Cl(\disc(g))$ together with a divisor $d$ of $B$ which it represents as $k(z_0,t_0) = d$. Let $h \in [k^{-2}g]$. Following Proposition~\ref{prop:samplingnbfxya}, sample a uniformly random integer solution $(z,x_1,y_1)$ of $az + bB^2h(x_1, y_1) = n$, with $z>0$. Note that the local solvability at $a$ in Proposition~\ref{prop:samplingnbfxya} is satisfied for any $k$ (because $h$ is always in the same genus), so the deviation from uniformity only comes from the error term. 
\end{proof}

\subsection{Proof of Theorem~\ref{thm:solvingthecomplicatedequation}}
Recall that $a = \det(\rho\gamma)^2$ and $g = f^{\rho\gamma}$.
From the assumption on the solutions modulo $\disc(f)$, there exists an invertible $u \bmod \disc(f)$ represented by the genus of $f$ and such that $\chi\left(\frac{n - ua}{b}\right) \neq -1$ (an exhaustive search finds it in time polynomial in $\disc(f)$).
From Lemma~\ref{lem:wecansampleinX}, we can efficiently sample $T \in \mathscr X$ such that $\pi(T)$ is uniform in $X$. From Lemma~\ref{lem:samplinginSissufficient}, we are done as soon as $\pi(T) \in \mathscr S$. Indeed, such a $T$ gives a solution of
$$n = a z + b g(x_0,x_0) = \det(\gamma)^2\det(\rho)^2\ell + b(f^{\gamma})^{\rho}(x_0, y_0),$$
giving rise via Lemma~\ref{lemma:randomizingontheleft} to a solution of $\det(\gamma)f(s,t) + bf^\gamma(x, y) = n$.
Then, it only remains to prove that $\#X/\#\mathscr S$ is small.
%
On one hand, from Proposition~\ref{prop:samplingnbfxya},
we have $$\#X \ll \frac{n2^{\omega(a)}h}{B^2 b a |\disc(g)|^{1/2}} \ll \frac{n\log|a\disc(f)|}{B^2 b a},$$ where $h \ll 2^{-\omega(a)} |\disc(g)|^{1/2}\log|\disc(g)|$ is the number of classes in the genus of $g$. On the other hand, the local solvability of the equation modulo $\disc(g)$ together with Corollary~\ref{coro:coromaintheoremthatimpliesklpt} implies that there is a constant $c'$ such that
\begin{align*}
\#\mathscr S & \gg \sum_{\substack{u' \bmod \disc(f) \det(\rho\gamma)\\ u' \equiv u \bmod \disc(f)\\ (u',\det(\rho\gamma)) = 1}}\frac{n\left(1 + \chi\left(\frac{n - ua}{b}\right)\right)}{B^2 b a \phi(\disc(f) \det(\rho\gamma)) (\log n)^{c'}}
 =\frac{n\left(1 + \chi\left(\frac{n - ua}{b}\right)\right)}{B^2 b a \phi(\disc(f)) (\log n)^{c'}}.
 \end{align*}
Note that since $a = \det(\rho\gamma)^2$ is 
coprime to $n$, the condition
$\GCD(n-au',\det(\rho\gamma)) = 1$ (required for Corollary~\ref{coro:coromaintheoremthatimpliesklpt}) is satisfied for any $u'$.
The theorem follows. 
\qed


\subsection{Representing integers in special orders}
Theorem~\ref{thm:solvingthecomplicatedequation} has an immediate, but important corollary.
Recall that for any prime $p$, we denote by $\cO_0$ the special order in $B_{p,\infty}$ defined in Lemma~\ref{lem:specialorders}.

\begin{corollary}[GRH]\label{coro:repintspecialorder}
There is a constant $c$ and an algorithm $\mathscr A$ such that the following holds. For any prime $p$ and integer $n$ with $\log n > (\log p)^c$, if 
either
$\log n \geq \omega(n)^{c}$, or  
the prime divisors of $n$ are larger than $(\log p)^c$, then
the algorithm $\mathscr A$
finds an element $\alpha \in \cO_0$ of reduced norm $n$, and runs in expected polynomial time in $\log p$ and $\log n$. The output $\alpha$ is random with min-entropy at least $\Omega( \log n)$.
\end{corollary}
\begin{proof}
With notations as in Lemma~\ref{lem:specialorders},
the order $\cO_0$ contains the elements $1,\omega,j$ and $\omega j$, and
$$\Nrd(s + t\omega + xj + y\omega j) = f(s,t) + p f(x,y).$$
When $\GCD(n,p) = 1$, the result follows from Theorem~\ref{thm:solvingthecomplicatedequation} with $b = p$ and $\gamma$ the identity matrix. 
In the case where prime divisors of $n$ are at least $(\log p)^c$, we use that $\disc(f) = O((\log p)^2)$, which allows to satisfy Condition~\eqref{item:conditionprimedivisorsn}.
Since $\Nrd(j) = p$, the general result follows by multiplicativity of the reduced norm.
\end{proof}


\section{Solving the quaternion path problem}\label{sec:newKLPT}
\noindent
For the rest of the article we consider the quaternion algebra $B_{p,\infty}$, with basis $1,i,j,ij$, as defined in Section~\ref{sec:introquat}, and $\cO_0$ is the special maximal order defined in Lemma~\ref{lem:specialorders}.
In this section, we consider the $\QuatPath$ problem.
Since computing connecting ideals between two maximal orders is easy (see~\cite[Algorithm~3.5]{KV10}), it is sufficient to consider the following problem: given a maximal order $\cO$ in $B_{p,\infty}$, a left $\cO$-ideal $I$, and an integer $N$, find an ideal $J$ equivalent to $I$ such that $\Nrd(J) = N$. As noted in~\cite[Section~4.6]{KLPT14}, the general case reduces to the case $\cO = \cO_0$, so we focus on this special maximal order.


\subsection{Random walks between ideal classes}
As a first step towards a rigorous algorithm, we start by randomising the input ideal, thereby avoiding pathological cases.

\begin{definition}[Brandt graph]
Let $p$ be a prime number, and $\cO$ a maximal order in $B_{p,\infty}$. Let $I,J$ be two left $\cO$-ideals. We say $J$ is an \emph{$\ell$-neighbor} of $I$ if $J \subseteq I$ and $\Nrd(J) = \ell \Nrd(I)$. The \emph{$\ell$-Brandt graph} is the graph with vertices $\mathrm{Cls}(\cO)$ and an edge from $[I_i]$ to $[J]$ for each $\ell$-neighbor $J \subseteq I_i$, where $(I_i)_{i=1}^{\#\mathrm{Cls}(\cO)}$ is a list of ideal class representatives.
\end{definition}

Through the Deuring correspondence, the $\ell$-Brandt graph is isomorphic to the $\ell$-isogeny graph (up to the action of $\Gal(\F_{p^2}/\F_p)$). It is common in isogeny-based cyptography to compute random walks on these graphs: at each step on the walk, the current vertex is an elliptic curve $E$, one chooses uniformly at random one of the $\ell+1$ outgoing isogenies, and the next vertex is its target. Equivalently, given a left $\cO$-ideal $I$, one can choose uniformly at random one of the $\ell+1$ left submodules $M \subset I/\ell I$, and the next vertex is $M + \ell I$.

\begin{theorem}[\protect{\cite[Theorem~1]{GPS20}}]\label{thm:randomwalkBrandt}
Let $p$ be a prime number, and $\cO$ a maximal order in $B_{p,\infty}$. Let $N_p$ be the size of the ideal class set of $\cO$. Let $I$ be the ideal obtain from a random walk of norm $n = \prod_{i}\ell_i^{e_i}$. Then, for any ideal class $C$, we have
$$\left| \Pr[I \in C] - \frac 1 {N_p}\right| \leq \prod_i\left(\frac{2\sqrt{\ell_i}}{\ell_i + 1}\right)^{e_i}.$$
\end{theorem}
\begin{proof}
This is precisely~\cite[Theorem~1]{GPS20}, translated from isogenies to quaternions through the Deuring correspondence. It is a consequence of the fact that each $\ell_i$-Brandt graph (or $\ell_i$-isogeny graph) has the Ramanujan property.
\end{proof}

\subsection{Solving the quaternion analog of the isogeny-path problem}
The main result of this section is the following theorem.
\begin{theorem}[GRH]\label{thm:rigorousklpt}
There exists an integer $c$ such that Algorithm~\ref{alg:EquivIdeal} is correct and runs in expected polynomial time in $\log p$, $\log n_i$, $\log \Nrd(I)$ and $\ell$ for all inputs satisfying $\log n_i \geq (\log p)^c$, and $n_2\ell^e \not\equiv 2,4 \bmod 8$ for $e \in \{0,1\}$, and either $\log n_2 \geq \omega(n_2)^c$, or all prime divisors of $n_2\ell$ are larger than $(\log p)^c$.
\end{theorem}

\begin{algorithm}[h]
 \caption{$\EquivIdeal_c(I,n_1,n_2)$}\label{alg:EquivIdeal}
 \begin{algorithmic}[1]
 \REQUIRE {A left ideal $I$ in the special maximal order $\cO_0$, positive integers $n_1,n_2$, and a prime~$\ell$.}
 \ENSURE {An equivalent ideal $J$ of norm $n_1n_2$ or $n_1n_2\ell$.}
 
\STATE Define $R$, $\omega$ and $f$ as in Lemma~\ref{lem:specialorders}.
\WHILE {$\beta$ has not been found}
\STATE $I' \gets$ the endpoint $I' \subset I$ of a random walk of norm $n_1$ in the Brandt graph; \COMMENT{Theorem~\ref{thm:randomwalkBrandt}}
\label{state:randomizeidealclassset}
\STATE $(I'',\rho) \gets $ an ideal $I''$ equivalent to $I'$, of prime norm $N \in [p^c,p^{2c}]$ such that $\ell$ is a non-quadratic residue modulo $N$, and the element $\rho \in I'$ such that $I'' = I'\overline{\rho}/\Nrd(I')$; 
\COMMENT{Proposition~\ref{prop:algEquivPrimeIdealLargeNonQuad}}
\STATE \label{step:randomelemofnormNn2} $\gamma \gets $ an element $\gamma \in \cO_0$ such that $\Nrd(\gamma) = N$;
\COMMENT{Corollary~\ref{coro:repintspecialorder}}
\STATE $\beta \gets $ an element $\beta \in R$ such that $I'' = \cO_0 N + \cO_0 \gamma \beta j$ if it exists;\label{step:findingbeta}
\ENDWHILE
\STATE $\Gamma \gets $ a matrix in $M_{2\times 2}(\Z)$ such that $\Z\beta + RN = \left\{x+y\omega \mid (x,y)\in \Gamma \Z^2\right\}$;
\STATE \label{step:solvingthecomplicatedequation}$(s,t,x,y) \gets $ an integral solution of $N^2 f(s,t) + p f^\Gamma(x,y) = n_2\ell^e$ for some $e \in \{0,1\}$;
\COMMENT{Theorem~\ref{thm:solvingthecomplicatedequation}}
\STATE $(x',y') \gets \Gamma(x,y)$;
\STATE $\alpha \gets (s + t \omega)N + (x' + y' \omega)j$;
\STATE $\delta \gets \rho\gamma\alpha/N \in I' \subset I$;
\RETURN $J = I\overline{\delta}/\Nrd(I)$.
  \end{algorithmic}
 \end{algorithm}

\begin{proof}
The efficiency and correctness of most steps are already justified by the various results referred to in the comments of Algorithm~\ref{alg:EquivIdeal}. The constraints on $n_2$ come from Theorem~\ref{thm:solvingthecomplicatedequation}.
\\

\noindent \emph{Step~\ref{state:randomizeidealclassset}.} From Theorem~\ref{thm:randomwalkBrandt}, there is a constant $c$ such that if $\log n_1 \geq (\log p)^c$, then $I'$ is in any given class set with probability between $1/2N_p$ and $3/2N_p$, with $N_p$ the number of classes.\\

\noindent \emph{Step~\ref{step:randomelemofnormNn2}.} Corollary~\ref{coro:repintspecialorder} requires $N = \Nrd(I'')$ to be large enough; therefore, in constructing $I''$, we resort to Proposition~\ref{prop:algEquivPrimeIdealLargeNonQuad} rather that Theorem~\ref{thm:algEquivPrimeIdeal}. This issue is dealt with differently in~\cite{KLPT14}: they solve an equation of the form $\Nrd(\gamma) = Nn_3$ for some large enough $n_3$. Our approach has a theoretical advantage: Corollary~\ref{coro:repintspecialorder} ensures that $\gamma$ has large entropy, which allows to avoid corner cases in Step~\ref{step:findingbeta}.
More precisely, let us prove that $\cO_0\gamma / \cO_0N$ has large entropy. It is sufficient to prove that the map $\gamma \mapsto \cO_0\gamma / \cO_0 N$, for $\gamma \in \cO_0$ of norm $N$, has small fibre. Suppose that $\cO_0\gamma / \cO_0 N = \cO_0\gamma' / \cO_0 N$. Then, there exists $x,y \in \cO_0$ such that $\gamma' = x\gamma + yN$. Then,
\begin{align*}
\gamma' = x\gamma + yN = (x + y\overline{\gamma})\gamma.
\end{align*}
Comparing norms, we deduce $\Nrd(x+y\overline{\gamma}) = 1$, hence $\gamma' \in \cO_0^\times\gamma$. Since $\#\cO_0^\times \leq 6$, the map $\gamma \mapsto \cO_0\gamma / \cO_0 N$ is $O(1)$-to-1, which proves that $\cO_0\gamma / \cO_0 N$ has large entropy.
\\

\noindent \emph{Step~\ref{step:findingbeta}.} 
This step is solved with elementary linear algebra, as described in~\cite[Section~4.3]{KLPT14}. The method of~\cite[Section~4.3]{KLPT14} succeeds under the assumption that $I''/N\cO_0$ and $\cO_0\gamma/N\cO_0$ are distinct from the (at most two) fixed points for the action of $(R/NR)^\times$. This is heuristically assumed in~\cite{KLPT14}, but with our new methods, we can prove it. The large entropy of $\cO_0\gamma/N\cO_0$ ensures that with good probability, it is not one of the two fixed points.
From Step~\ref{state:randomizeidealclassset}, $I'$ is close to uniformly distributed in the class set, so with overwhelming probability it is not equivalent to an ideal induced by an $R$-ideal (i.e., to an ideal of the form $\cO_0\mathfrak a$ for some $R$-ideal $\mathfrak a$). It is then also the case of $I''$, so with good probability, it is not a fixed point either (the ideals of norm $N$ that are fixed points are induced by the $R$-ideals above $N$).\\

\noindent \emph{Step~\ref{step:solvingthecomplicatedequation}.}
Most conditions for Theorem~\ref{thm:solvingthecomplicatedequation} are already met.
The value $e \in \{0,1\}$ is determined by the constraint that the equation must have a solution modulo $N$.
It remains to justify that the equation does have a solution in $G = \Z/\disc(f)\Z$ for which $f(s,t)$ is invertible.
Suppose $p\equiv 1 \bmod 8$, so from Lemma~\ref{lem:quaternionpq}, $\disc(f)$ is a negative odd prime.
From~\cite[Theorem 3.15]{Cox11}, there is only one genus of forms of discriminant $\disc(f)$, so $f$ represents all quadratic residues in $G$. Since $(N, \disc(f))=1$, the form $f^\Gamma$ also represents all the quadratic residues in $G$. Lemma~\ref{lem:quaternionpq} implies that $p$ is not a quadratic residue. Any element is $G$ is the sum of a quadratic residue and a quadratic non-residue, so $n_2\ell^e$ also is, and we are done.
Similarly, if $\disc(f) = 4$ there is a solution when $n_2\ell^e \not\equiv 2 \bmod 4$, and if $\disc(f) = 8$ there is a solution when $n_2\ell^e \not\equiv 4 \bmod 8$.
\end{proof}

\begin{remark}
Given a prime $\ell$,
one can choose $n_1$ and $n_2$ to be large enough powers of $\ell$, so
Algorithm~\ref{alg:EquivIdeal} straightforwardly specialises to the power-of-$\ell$ variant $\ell$-$\QuatPath$.
We deal with the powersmooth variant $B$-$\PSQuatPath$ in the next section.
\end{remark}

\subsection{Finding power-smooth paths}
In Theorem~\ref{thm:rigorousklpt}, the integers $n_i$ either have very few prime factors, or the prime factors are not too small. This seems to come at odds with a major application of~\cite{KLPT14}: constructing ideals of powersmooth norm. We now prove that it is not an issue, and we can indeed solve the $B$-$\PSQuatPath$ variant.

\begin{theorem}[GRH]\label{thm:powersmoothclassrep}
There exists an integer $c$ and an algorithm $\mathscr A$ such that the following holds. On input a left $\cO_0$-ideal $I$, 
the algorithm outputs an equivalent ideal $J$ whose norm is $(\log p)^c$-powersmooth, and runs in expected polynomial time in $\log p$ and $\log \Nrd(I)$.
\end{theorem}

\begin{proof}
It is sufficient to prove that one can find suitable powersmooth integers $n_i$ to apply Theorem~\ref{thm:rigorousklpt}. 
Let $c_0$ be the constant from Theorem~\ref{thm:rigorousklpt}, and let $\delta>0$ be some parameter to be adjusted. Let $c = 2(c_0 + \delta)$.
We need to construct two $(\log p)^{c_0+\delta}$-powersmooth integers $n_i$  such that $\log n_i > (\log p)^{c_0}$ and prime divisors of $n_2$ are larger than $(\log p)^{c_0}$.
We choose
$$n_1 = n_2 = \prod_{(\log p)^{c_0} < \ell < (\log p)^{c_0+\delta}} \ell ^{\frac{(c_0+\delta)\log p}{\log \ell}}.$$
Then, from the prime number theorem (with Riemann's hypothesis),
\begin{align*}
\log n_2 \geq \sum_{(\log p)^c_0 < \ell < (\log p)^{c_0+\delta}} \log \ell
= (\log p)^{c_0}\left((\log p)^\delta - 1\right) + O\left((\log p)^{\frac{c_0+\delta}{2}}(\log \log p)^2\right).
\end{align*}
Choosing $\delta$ large enough ensures that $\log n_2 > (\log p)^{c_0}$, which concludes the proof.
\end{proof}

\section{Maximal Order and Isogeny Path are equivalent}\label{sec:MOequivPath}
\noindent
In this section and the next, we prove that $\EllIsoPath$, $\EndRing$ and $\MaxOrder$ are all equivalent, under probabilistic polynomial-time reductions. We start in this section by showing that $\EllIsoPath$ is equivalent to $\MaxOrder$.

\subsection{Maximal Order reduces to Isogeny Path}
From Lemma~\ref{lem:idealtoisogeny}, we know how to translate powersmooth $\cO_0$-ideals into isogenies.
The following lemma deals with the converse direction.
\begin{lemma}\label{lem:isogenytoideal}
Let $\cO_0$ and $E_0$ as in Lemmata~\ref{lem:specialorders} and~\ref{lem:constructingE0}.
There exists an algorithm which, given an isogeny $\varphi: E_0 \rightarrow E$ of degree $\prod_i \ell_i^{e_i}$, returns the corresponding left $\cO_0$-ideal $I_\varphi$. The complexity of this algorithm is polynomial in $\log p$ and $\max_i (\ell_i^{e_i})$ (if $p \equiv 1 \bmod 8$, we assume GRH).
\end{lemma}
\begin{proof}
A proof of this lemma was first given in~\cite{KrigsmanThesis}, building upon the heuristic result~\cite[Lemma~6]{GPS20}. It can also be seen as a consequence of Lemma~\ref{lem:idealtoisogeny}: for each $i$,
\begin{enumerate}
\item Enumerate the set $S_i$ of all left $\cO_0$-ideals of norm $\ell_i^{e_i}$ (see~\cite{KV10});
\item For each $J \in S_i$, compute the corresponding isogeny $\varphi_J$ with Lemma~\ref{lem:idealtoisogeny}, and if $\ker (\varphi_J) = \ker(\varphi) \cap E_0[\ell_i^{e_i}]$, let $I_i = J$.
\end{enumerate}
Finally, return $I_\varphi = \bigcap_i I_i$. Of course, this guessing approach is not as efficient as the method proposed in~\cite[Lemma~6]{GPS20}, but it is still polynomial in $\max_i (\ell_i^{e_i})$.
\end{proof}

To prove that $\MaxOrder$ reduces to $\EllIsoPath$, we show that an isogeny between $E$ and the special curve $E_0$ (of known endomorphism ring) allows to recover the endomorphism ring of $E$. A heuristic version of this approach was described in~\cite{DFMPS19}.

\begin{algorithm}[h]
 \caption{Reducing $\MaxOrder$ to $\EllIsoPath$}\label{alg:EndToPath}
 \begin{algorithmic}[1]
 \REQUIRE {A supersingular elliptic curves $E/\F_{p^2}$, with $p \neq \ell$. We suppose there is an algorithm $\mathscr A_{\EllIsoPath}$ that solves the $\ell$-$\IsoPath$ problem.}
 \ENSURE {A basis of an order in $B_{p,\infty}$ isomorphic to $\End(E)$.}
  \STATE $c \gets$ the constant from Theorem~\ref{thm:powersmoothclassrep};
\STATE $(\cO_0,E_0) \gets$ the special order and curve from Lemmata~\ref{lem:specialorders} and~\ref{lem:constructingE0};
\STATE $\varphi \gets \mathscr A_{\EllIsoPath}(E_0,E)$, with $\varphi = \varphi_e\circ\dots\circ\varphi_1$, and $\deg(\varphi_i) = \ell$;
\STATE $\psi_0 \gets$ the identity isogeny $E_0 \rightarrow E_0$;
\FOR{$i = 1,\dots,e$}
\STATE $I_i \gets$ the ideal corresponding to $\varphi_i \circ \psi_{i-1}$; \COMMENT{Lemma~\ref{lem:isogenytoideal}}
\STATE $J_i \gets $ an ideal equivalent to $I_i$, with $(\log p)^c$-powersmooth norm; \COMMENT{Theorem~\ref{thm:powersmoothclassrep}}
\STATE $\psi_i \gets $ the isogeny corresponding to $J_i$; \COMMENT{Lemma~\ref{lem:idealtoisogeny}}
\ENDFOR
\STATE $\cO \gets \cO_R(J_e)$; 
\COMMENT{\cite[Theorem~3.2]{Ron92}}
\RETURN A basis of $\cO$.
  \end{algorithmic}
 \end{algorithm}

\begin{theorem}[GRH]\label{thm:EndToPath}
The reduction in Algorithm~\ref{alg:EndToPath} is correct and runs in expected polynomial time in $\log p$ and the output size of $\mathscr A_{\EllIsoPath}$, plus one call to $\mathscr A_{\EllIsoPath}$.
\end{theorem}

\begin{proof}
Write $\varphi_i : E_{i-1} \rightarrow E_i$, with $E_e = E$. 
At each step of the loop, we have that $J_i$ is equivalent to the ideal corresponding to $\varphi_i\circ\dots\circ\varphi_1$, 
hence $\cO_R(J_i) \cong \End(E_i)$. This proves the correctness. The running time follows from the results cited at each step of Algorithm~\ref{alg:EndToPath}.
\end{proof}

\subsection{Isogeny Path reduces to Maximal Order}
To reduce $\EllIsoPath$ to $\MaxOrder$, we prove that one can translate left $\cO_0$-ideals of norm a power of $\ell$ to the corresponding isogeny.

\begin{algorithm}[h]
 \caption{Translating a left $\cO_0$-ideal  of prime-power norm to an isogeny}\label{alg:idealtoisogeny}
 \begin{algorithmic}[1]
 \REQUIRE {A left $\cO_0$-ideal $I$ of norm $\ell^e$, with $\ell \neq p$ prime, and $\ell \nmid I$.}
 \ENSURE {The corresponding isogeny $\varphi_I$.}
  \STATE $c \gets$ the constant from Theorem~\ref{thm:powersmoothclassrep};
 \FOR{$i = 1,\dots,e$}
 \STATE $I_{i} \gets I + \cO_0 \ell^{i}$;
 \STATE $J_{i} \gets $ an ideal equivalent to $I_{i}$, with $(\log p)^c$-powersmooth norm; \COMMENT{Theorem~\ref{thm:powersmoothclassrep}}
 \STATE $\psi_{i} \gets $ the isogeny corresponding to $J_{i}$; \COMMENT{Lemma~\ref{lem:idealtoisogeny}} 
 \STATE $E_{i} \gets \mathrm{target}(\psi_{i})$;
 \STATE $\varphi_{i} \gets$ the $\ell$-isogeny from $E_{i-1}$ to $E_{i}$; \COMMENT{see~\cite{velu}}
\ENDFOR
\RETURN $\varphi_e \circ \dots \circ \varphi_1$.
  \end{algorithmic}
 \end{algorithm}

\begin{lemma}\label{lem:idealtoisogenypower}
Algorithm~\ref{alg:idealtoisogeny} is correct and runs in  expected polynomial  time  in $\log p$, $\ell$ and $e$ (if $p \equiv 1 \bmod 8$, we assume GRH).
\end{lemma}

\begin{proof}
Heuristic versions of this strategy have already appeared in the literature (for instance as a part of~\cite[Algorithm~7]{EHLMP18}). Using Theorem~\ref{thm:powersmoothclassrep} instead of~\cite{KLPT14} makes it rigorous.
\end{proof}

\begin{algorithm}[h]
 \caption{Reducing $\EllIsoPath$ to $\MaxOrder$}\label{alg:PathToMax}
 \begin{algorithmic}[1]
 \REQUIRE {Two supersingular elliptic curves $E_1$ and $E_2$ over $\F_{p^2}$. We suppose we are given the two $\MaxOrder$-solutions
 $\cO_1$ and $\cO_2$, maximal orders in $B_{p,\infty}$ isomorphic to $\End(E_1)$ and $\End(E_2)$ respectively.
 }
 \ENSURE {An $\ell$-isogeny path from $E_1$ to $E_2$.}
 \STATE $c \gets$ the constant from Theorems~\ref{thm:rigorousklpt};
\STATE $e \gets \left\lceil \frac{(\log p)^{c}}{\log \ell}\right\rceil$;
\STATE $(\cO_0,E_0) \gets$ the special order and curve from Lemmata~\ref{lem:specialorders} and~\ref{lem:constructingE0};
 \FOR{$i = 1,2$}
\STATE $I_i \gets I(\cO_0,\cO_i)$ the ideal connecting $\cO_0$ and $\cO_i$; \COMMENT{\cite[Algorithm~3.5]{KV10}}
\STATE $J_i \gets \EquivIdeal_c(I_i,\ell^e,\ell^e,\ell)$; \COMMENT{Theorem~\ref{thm:rigorousklpt}}
\STATE $\varphi_i \gets$ the isogeny corresponding to $J_i$; \COMMENT{Lemma~\ref{lem:idealtoisogenypower}}
\ENDFOR
\RETURN $\varphi_2 \circ \hat\varphi_1$.
  \end{algorithmic}
 \end{algorithm}

 \begin{theorem}[GRH]\label{thm:PathToMax}
Algorithm~\ref{alg:PathToMax} is correct and runs in expected polynomial time in $\log p$, $\ell$ and in the length of the two provided $\MaxOrder$-solutions $\cO_1$ and $\cO_2$.
 \end{theorem}
 
 \begin{proof}
The reduction is almost the same as~\cite[Algorithm~7]{EHLMP18}, but using Theorems~\ref{thm:rigorousklpt} and~\ref{thm:powersmoothclassrep} instead of~\cite{KLPT14}. Note that we reduce to $\MaxOrder$ whereas~\cite[Algorithm~7]{EHLMP18} reduces to $\EndRing$. However, in~\cite[Algorithm~7]{EHLMP18}, the algorithm solving $\EndRing$ is only used to recover the maximal orders $\cO_1$ and $\cO_2$ via the reduction from $\MaxOrder$ to $\EndRing$. We simply short-circuit the chain of reductions.
 \end{proof}

%


\section{Endomorphism Ring is equivalent to Maximal Order}\label{sec:MOequivEnd}
\noindent We finally prove the equivalence between $\EndRing$ and $\MaxOrder$.

\subsection{Endomorphism Ring reduces to Maximal Order} We start with the simplest direction, which can readily be adapted from previous heuristic reductions with our new tools.

\begin{algorithm}[h]
 \caption{Reducing $\EndRing$ to $\MaxOrder$, with a parameter $\delta > 0$}\label{alg:EndToMO}
 \begin{algorithmic}[1]
 \REQUIRE {A supersingular elliptic curve $E/\F_{p^2}$, with $p \neq \ell$. We suppose we are given 
 the $\MaxOrder$-solution $\cO$, a maximal order in $B_{p,\infty}$ isomorphic to $\End(E)$. 
 }
 \ENSURE {Four endomorphisms of $E$ that generate $\End(E)$.}

\STATE $I \gets I(\cO_0,\cO)$, the ideal connecting $\cO_0$ to $\cO$; \COMMENT{\cite[Algorithm~3.5]{KV10}}
\STATE $c \gets$ the constant from Theorem~\ref{thm:powersmoothclassrep};
\STATE $J \gets $ an ideal equivalent to $I$, with $(\log p)^c$-powersmooth norm; \COMMENT{Theorem~\ref{thm:powersmoothclassrep}}
\STATE $\cO' \gets \cO_R(J)$ the right-order of $J$; 
\COMMENT{\cite[Theorem~3.2]{Ron92}}
\STATE $(\beta_i)_{i=1}^4 \gets$ a basis of $\cO'$;
\STATE $(\alpha_i)_{i=1}^4,(\phi_i)_{i=1}^4 \gets $ the special basis $(\alpha_i)_{i=1}^4$ of $\cO_0$ from Lemma~\ref{lem:specialorders}, with the corresponding endomorphisms $\phi_i \in \End(E_0)$ from Lemma~\ref{lem:constructingE0};
\STATE $(c_{ij})_{i,j=1}^4 \gets$ integers such that $\Nrd(J)\beta_i = \sum_{j=1}^4c_{ij}\alpha_j$ for $i = 1,\dots,4$;
 \STATE $\varphi \gets $ the isogeny corresponding to $J$; \COMMENT{Lemma~\ref{lem:idealtoisogeny}} 
\RETURN $(N,\varphi,(c_{ij})_{i,j})$, which represents the endomorphisms $\frac 1 N \sum_{j=1}^4c_{ij} \varphi \phi_j \hat \varphi$.
  \end{algorithmic}
 \end{algorithm}

\begin{theorem}[GRH]\label{thm:EndToMO}
Algorithm~\ref{alg:EndToMO} is correct and runs in expected polynomial time in $\log p$ and in the length of the provided $\MaxOrder$-solution $\cO$.
\end{theorem}

\begin{proof}
The algorithm is the same as~\cite[Algorithm~4]{EHLMP18}, but using Theorem~\ref{thm:rigorousklpt} instead of~\cite{KLPT14}. In particular, it is proven in~\cite[Lemma~3]{EHLMP18} that $(N,\varphi,(c_{ij})_{i,j})$ is an efficient representation of the basis.
\end{proof}

\subsection{Maximal Order reduces to Endomorphism Ring}
Finally, we prove that $\MaxOrder$ reduces to $\EndRing$.
The most delicate issue is that the corresponding heuristic reduction \cite[Algorithm~6]{EHLMP18} requires the factorisation of large integers, a task that in the worst case cannot be solved in polynomial time (to the best or our knowledge). We modify the reduction to provably avoid all hard factorisations. To do so, we force the corresponding integers to be prime, by leveraging Proposition~\ref{prop:samplingprimegeneral} and an explicit parameterisation of solutions of quadratic forms.
We start with a lemma, and introducing some handy notation.
\begin{lemma}\label{lem:evalscalarproduct}
Given two endomorphisms $\alpha$ and $\beta$ in an efficient representation, one can compute $\langle \alpha, \beta\rangle$ 
in time polynomial in the length of the representation of $\alpha$ and $\beta$, and in $\log p$.
\end{lemma}
\begin{proof}
This is proven in \cite[Lemma~4]{EHLMP18}. Recall that an efficient representation means that there is an algorithm that evaluates $\alpha(P)$ for any $P \in E(\F_{p^k})$ in time polynomial in the length of the representation of $\alpha$ and in $k\log p$. Also, the length of an efficient representation of $\alpha$ is $\Omega(\log (\deg(\alpha)))$ (which rules out exotic representations where the number of bits of $\langle \alpha, \beta\rangle$ would be exponential in the length of the input).
\end{proof}

\begin{notation}
Given two quadratic forms $f$ and $g$, we write $f\orthsum g$ their orthogonal sum, defined as
$$(f\orthsum g)(x,y) = f(x) + g(y).$$
We extend this notation naturally to $U \orthsum V$ or $G\orthsum H$ for quadratic spaces $U$ and $V$ or Gram matrices $G$ and $H$.
\end{notation}

\begin{notation}
We write $\langle a_1,\dots,a_r\rangle$ the quadratic form whose Gram matrix is $\diag(a_1,\dots,a_r)$.
\end{notation}

\begin{algorithm}[h]
 \caption{Reducing $\MaxOrder$ to $\EndRing$}\label{alg:MOToEnd}
 \begin{algorithmic}[1]
 \REQUIRE {A supersingular elliptic curve $E/\F_{p^2}$. 
 We suppose we are given the $\EndRing$-solution $(\beta_i)_{i=1}^4$, a list of four endomorphisms that generate $\End(E)$.
 }
 \ENSURE {A basis of an order in $B_{p,\infty}$ isomorphic to $\End(E)$.}
\STATE \label{step:MOToEnd:2} $G_0 \gets (\langle \beta_i,\beta_j \rangle)_{i,j=1}^4$ the Gram matrix of $(\beta_i)_{i=1}^4$;
\STATE \label{step:MOToEnd:3} \label{step:MOToEnd:3} Find a change of basis such that $A^tG_0A = \langle 1\rangle \orthsum G$, where $G$ is integral and $\disc(G)$ is only divisible by $p$ and $2$;
\STATE \label{step:MOToEnd:4} Solve $x^t G x=q(a\ell)^2$, where $x \in \Z^3$ is primitive, $\ell$ is prime (or $\ell = 1$) and $a$ may only be divisible by the primes $2$ and $p$;
\STATE \label{step:MOToEnd:5} Find a change of basis $B = \langle 1 \rangle \orthsum B'$ such that $B^t (\langle 1\rangle \orthsum G) B = \langle 1,q \rangle \orthsum H$, where $H$ is integral, and $\disc(H)$ is only divisible by $2, p, q$ and $\ell$.
\STATE \label{step:MOToEnd:6} Solve $y^tHy = p$ with $y \in \Q^2$;
\STATE \label{step:MOToEnd:7} $\iota \gets \left(A\left(0, \frac{x}{a\ell}\right)\right)^t (\beta_i)_{i=1}^4$;
\STATE \label{step:MOToEnd:8} $\pi \gets \left(AB (0, 0, y)\right)^t(\beta_i)_{i=1}^4$;
\STATE \label{step:MOToEnd:9} $\kappa \gets \iota\circ \pi$;
\STATE \label{step:MOToEnd:10} $\Phi : \End(E) \otimes \Q \rightarrow B_{p,\infty}$, the isomorphism sending $1,\iota,\pi,\kappa$ to $1,i,j,ij$;
\RETURN \label{step:MOToEnd:11} $(\Phi(\beta_i))_{i = 1}^4$.
  \end{algorithmic}
 \end{algorithm}

 \begin{theorem}[GRH]\label{thm:MOToEnd}
Algorithm~\ref{alg:MOToEnd} is correct and runs in expected polynomial time in $\log p$, and in the length of the provided $\EndRing$-solution $(\beta_i)_{i=1}^4$.
 \end{theorem}
\begin{proof}
Let us go through the reduction step by step.\\

\noindent\emph{Step~\ref{step:MOToEnd:2}.} The Gram matrix $G_0$ of $(\beta_1,\beta_2,\beta_3,\beta_4)$ can be computed via Lemma~\ref{lem:evalscalarproduct}. \\

\noindent\emph{Step~\ref{step:MOToEnd:3}.}  First recall that $\disc(G_0) = \disc(\End(E)) = p^2$. This step follows from the fact that the endomorphisms $(2\beta_i - \tr(\beta_i))_{i=1}^4$ generate the (rank $3$) orthogonal complement of $1$ in $\Z + 2\End(E)$ (an order of discriminant only divisible by $p$ and $2$).
\\ 

\noindent\emph{Step~\ref{step:MOToEnd:4}.} This step calls for more extensive explanations. 
First note that the norm form on $B_{p,\infty}$ is $\Q$-equivalent to $\langle 1,q,p,qp\rangle$, so by the cancellation theorem, $G \simeq_{\Q} \langle q,p,qp\rangle$.
Let $Q = G \orthsum \langle -q \rangle$. The factorisation of $\disc(G)$ (hence $\disc(Q)$) being known, we can find a solution $X_0^t Q X_0=0$  with $X_0 = (x_0,\ell_0)$,  where $x_0 \in \Z^3$ is primitive and $\ell_0 \in \Z_{>0}$ using~\cite{Simon06}.
Yet, $\ell_0$ is not necessarily prime. From~\cite[Proposition~6.3.2]{Cohen08}, the general solution $X$ is given by
$$X  = d((R^tQR)X_0 - 2(R^tQX_0)R),$$
for arbitrary $R \in \Q^4$ and $d \in \Q^*$. Fix $d = 1$. Write $R = (r_x,r_\ell)$ with $r_x \in \Z^3$ and $r_\ell \in \Z$. The last coordinate of $X$ is given by the integral quadratic form
$$r_x^t G r_x \ell_0 - 2 r_x^tGx_0r_\ell + q\ell_0 r_\ell^2 = \frac{(r_x\ell_0 - x_0r_\ell)^tG(r_x\ell_0 - x_0r_\ell)}{\ell_0}.$$
It is of rank $3$, so let $M \in M_{3\times 3}(\Z)$ be a matrix whose columns generate $\Lambda = \ell_0\Z^3 + x_0 \Z$, and
$$g(z) = \frac{z^t(M^tGM)z}{\ell_0}.$$
It is positive definite, since $G$ is and $\ell_0 > 0$. 
Let us show that $g$ is (almost) primitive. If $s$ is a prime that does not divide $\ell_0$, both $M$ and $\ell_0$ are invertible modulo $s$, so $g$ is primitive at $s$ because $G$ is. Now suppose $s \mid \ell_0$. Then, writing $Mz = r_x\ell_0 - x_0r_\ell$, we have
\begin{align*}
g(z) \equiv  - 2 r_x^tGx_0r_\ell \bmod s.
\end{align*}
Therefore, if $s \neq 2$ and $Gx_0 \not\equiv 0 \mod s$, then $g$ is primitive at $s$.
If $Gx_0 \equiv 0 \mod s$, since $x_0$ is primitive, $s$ must divide $\disc(G)$, so $s$ is $2$ or $p$.
This proves that the only primes where $g$ might not be primitive are $2$ and $p$.
We can then write $g = g'/a$ where $g'$ is primitive and $a$ may only be divisible by the primes $2$ and $p$. Applying Proposition~\ref{prop:samplingprimegeneral}, we can find in polynomial time a $z$ such that $\ell = g'(z)$ is prime, hence a solution  of the form $x^t G x=q(a\ell)^2$. In this solution, we can assume that $x$ is primitive: if $c$ divides the content of $x$, then $c^2$ divides $q(a\ell)^2$, so $c$ divides $a\ell$.\\

\noindent\emph{Step~\ref{step:MOToEnd:5}.} We are looking for $B'$ such that $(B')^tGB' = \langle q \rangle \orthsum H$. Let $ (x \mid \Gamma)$ be a unimodular integral matrix with first column equal to $x$ (it can be found because $x$ is primitive). Let
$$P = I_3 - \left(\begin{matrix}
\frac{ e_1^t G x }{x^t G x} x \mid \frac{ e_2^t G x }{x^t G x} x \mid \frac{ e_3^t G x }{x^t G x} x\\
\end{matrix}\right)$$
be the $3\times 3$ matrix projecting orthogonally along $x$. With $B' = (x/(a\ell) \mid (x^t G x)P\Gamma)$, we obtain $(B')^tGB'$ of the desired form.\\

\noindent\emph{Step~\ref{step:MOToEnd:6}.} It can be solved efficiently with~\cite{Simon05} since the factorisation of $\disc(H)$ is known.\\

\noindent\emph{Steps~\ref{step:MOToEnd:7} to~\ref{step:MOToEnd:10}.} In $\End(E) \otimes \Q$, we have $\Nrd(\iota) = q$ and $\tr(\iota) = 0$ so $\iota^2 = -q$. Similarly, $\pi^2 = - p$. Therefore $\Phi$ is indeed an isomorphism. \\

\noindent\emph{Step~\ref{step:MOToEnd:11}.} 
All we need to do is express each $\beta_i$ in the basis $1,\iota,\pi,\kappa$ (allowing to evaluate $\Phi(\beta_i)$). We already know how to express  $1,\iota,\pi$ in terms of $(\beta_i)_{i = 1}^4$; if we can also express $\kappa$, then we obtain a change of basis between $1,\iota,\pi,\kappa$ and $(\beta_i)_{i = 1}^4$ and we are done.
Without loss of generality, $\beta_4$ is not in $\mathrm{span}(1,\iota,\pi)$. Let
$$\gamma = \beta_4 - \langle\beta_4, 1\rangle - \langle\beta_4, \iota\rangle \iota - \langle\beta_4, \pi\rangle \pi.$$
Then, $\gamma$ is orthogonal to $\mathrm{span}(1,\iota,\pi)$, so it belongs to $\mathrm{span}(\kappa)$. Renormalising, we obtain $\kappa$ as a combination of $(\beta_i)_{i = 1}^4$.
\end{proof}

\section{
Acknowledgements}
\noindent
The author wishes to thank Corentin Perret-Gentil and L\'eo Ducas for their help and feedback on several aspects of this work. This work was supported by the Agence Nationale de la Recherche under grants ANR MELODIA (ANR-20-CE40-0013) and ANR CIAO (ANR-19-CE48-0008).

\bibliographystyle{alpha}
\bibliography{bib}

\end{document}